\documentclass[a4paper,11pt]{amsart}

\usepackage{hyperref}
\usepackage{amssymb, amsthm, amsmath, thmtools, mathtools}

\usepackage{tikz}

\usepackage{bbm}

\usepackage[utf8x]{inputenc}

\usepackage[english]{babel}
\usepackage[abbrev]{amsrefs}

\usepackage{tikz}
\usepackage{tikz-cd}
\usetikzlibrary{matrix,arrows, arrows.meta,calc}
\usepackage{pgfplots}
\pgfplotsset{width=7cm,compat=1.8}
\usepgfplotslibrary{polar}

\allowdisplaybreaks[1]

\usepackage{enumitem}
\usepackage{verbatim}

\usepackage{comment}

\excludecomment{extracalc}

\declaretheorem[name=Theorem,numberwithin=section]{theorem}
\declaretheorem[name=Remark,style=remark,sibling=theorem]{remark}
\declaretheorem[name=Lemma,sibling=theorem]{lemma}
\declaretheorem[name=Proposition,sibling=theorem]{proposition}

\declaretheorem[name=Theorem,numbered=no]{theorem*}
\declaretheorem[name=Corollary,numbered=no]{corollary*}
\declaretheorem[name=Question,numbered=no]{question*}
\declaretheorem[style=remark,name=Remark,numbered=no]{remark*}
\declaretheorem[style=definition,name=Definition,numbered=no]{definition*}
\declaretheorem[style=definition,name=Notation,numbered=no]{notation*}
\declaretheorem[name=Lemma,numbered=no]{lemma*}

\newcounter{mt}

\declaretheorem[name=Theorem,sibling=mt]{MainTheorem}
\declaretheorem[name=Corollary,sibling=mt]{MainCorollary}

\makeatletter
\newcommand{\tpitchfork}{%
	\vbox{
		\baselineskip\z@skip
		\lineskip-.52ex
		\lineskiplimit\maxdimen
		\m@th
		\ialign{##\crcr\hidewidth\smash{$-$}\hidewidth\crcr$\pitchfork$\crcr}
	}%
}
\makeatother

\BibSpec{article}{%
	+{}{\PrintAuthors}  		{author}
	+{,}{ \textit}     		{title}
	+{,}{ }             		{journal}
	+{}{ \textbf}       		{volume}
	+{}{ \parenthesize} 		{date}
	+{}{, no. } 		{number}
	+{,}{ }      	      		{conference}
	+{,}{ }      	      		{book}
	+{,}{ }            		{pages}
	+{,}{ }            	 	{note}
	+{,}{ }            	 	{status}
	+{,}{  \texttt } {eprint}
	+{.}{}              {transition}
}

\BibSpec{book}{%
	+{}{\PrintAuthors}  {author}
	+{,}{ \textit}      {title}
	+{,}{ }             {publisher}
	+{,}{ }             {place}
	+{,}{ }             {date}
	+{.}{}              {transition}
}

\BibSpec{incollection}{%
	+{}  {\PrintAuthors}                {author}
	+{,} { \textit}                     {title}
	+{.} { }                            {part}
	+{:} { \textit}                     {subtitle}
	+{,} { \PrintContributions}         {contribution}
	+{,} { \PrintConference}            {conference}
	+{,} { }                            {booktitle}
	+{,} { pp.~}                        {pages}
	+{,}{ }       		{series}
	+{}{ \textbf}       		{volume}
	+{,} { }                            {publisher}
	+{,} { }                 {date}
	+{,} { }                            {status}
	+{,} { \PrintDOI}                   {doi}
	+{,} { available at \eprint}        {eprint}
	+{}  { \parenthesize}               {language}
	+{}  { \PrintTranslation}           {translation}
	+{;} { \PrintReprint}               {reprint}
	+{.} { }                            {note}
	+{.} {}                             {transition}
}

\usepackage{xcolor}
\newcommand{\NN}{\mathbb{N}}

\newcommand{\RR}{\mathbb{R}}

\newcommand{\ZZ}{\mathbb{Z}}
\newcommand{\QQ}{\mathbb{Q}}

\newcommand{\torus}[1]{\RR^{#1}/\ZZ^{#1}}

\newcommand{\tcount}{\mathrm{t}}

\newcommand{\hyp}{\mathcal{H}}

\newcommand{\hypcells}{\hyp^{\circ}}

\newcommand{\eins}{\mathbf{1}}

\DeclareMathOperator{\relint}{relint}
\DeclareMathOperator{\cl}{cl}
\newcommand{\polys}{\mathcal{P}}
\newcommand{\pc}{\mathsf{PC}}
\DeclareMathOperator{\bd}{bd}
\DeclareMathOperator{\tangent}{Tan}
\DeclareMathOperator{\normal}{Nor}
\DeclareMathOperator{\surf}{S}

\DeclareMathOperator{\pos}{pos}
\DeclareMathOperator{\aff}{aff}
\DeclareMathOperator{\spann}{span}
\newcommand{\bdproj}{\mathcal{B}}

\DeclareMathOperator{\Proj}{p}

\newcommand{\haus}{\mathcal{H}}
\DeclareMathOperator{\hf}{h}

\newcommand{\pair}[2]{\left\langle #1, #2 \right\rangle}

\DeclareMathOperator{\vol}{vol}

\DeclareMathOperator{\Tan}{Tan}

\DeclareMathOperator{\interior}{int}

\newcommand{\K}{\mathcal{K}}

\newcommand{\functions}{\mathcal{F}}
\newcommand{\kubert}{\mathrm{K}}

\title[Counting Lattice Points in Translations]{Counting Lattice Points in Translations: Uniqueness and Existence Results}
\author{Ansgar Freyer}
\address{FU Berlin, AG Diskrete Geometrie und Topologische Kombinatorik, Arnimallee 2, 14195 Berlin}
\email{a.freyer@fu-berlin.de}
\author{Georg C. Hofst\"atter}
\address{TU Wien, Institute f. Discrete Mathematics and Geometry, 1040 Vienna, Austria}
\email{georg.hofstaetter@tuwien.ac.at}

\thanks{{\it MSC classification}: Primary:
11P21, 
52B20, 
52C07. 
Secondary:
	52B45, 
	52C35. 
	\\
	\indent{\it Keywords}: Lattice point enumerator, Convex bodies, Lattice polytopes, Toric arrangements, Kubert--Lang operators 
	\\\indent
	The research of the first-named author was supported by the SPP 2458 ``Combinatorial Synergies'', funded by the Deutsche Forschungsgemeinschaft (DFG, German Research Foundation) – 539867386. The research of the second-named author was funded in part by the Austrian Science Fund (FWF), 10.55776/ESP9378724.}

\begin{document}
	\begin{abstract}
		Generalizing a recent theorem by Higashitani et al.\ (Comb. Theory 2025), it is shown that a convex body $K$ is uniquely described up to integer translation by the function $\tcount_K(v) = |(K-v)\cap\ZZ^n|$, $v\in\RR^n$. For a lattice polytope $P$ the function $\tcount_P$ is decomposed into homogeneous components $\tcount_{P,i}$, similar to the Ehrhart polynomial. It is shown that the functions $\tcount_{P,i}$ are simultaneous eigenfunctions of a family of linear operators introduced by Kubert and Lang  in (Math.\ Z.\ 1976) and that, conversely, the linear spaces of piecewise constant eigenfunctions are generated by the respective coefficients $\tcount_{P,i}$. A structure theorem for the simultaneous eigenspace corresponding to the $(n-1)$st coefficient in the decomposition of $\tcount_P$ is obtained. This leads to an explicit description of the set of functions $\tcount_P$, where $P$ ranges over all lattice polygons, as well as to inequalities for the number of lattice points in dilated and translated lattice polytopes.
	\end{abstract}
	
	\maketitle
	

	\section{Introduction}	
	
	For a convex body $K\subset\RR^n$, $n \geq 1$, i.e., a convex compact set, we consider
	\begin{align}\label{eq:defTranslCount}
		\tcount_K : \RR^n\to\ZZ_{\geq 0},\quad\tcount_K(v) = |(K-v)\cap\ZZ^n|,
	\end{align}
which counts the number of integer points in $K-v$. We call $\tcount_K$ the \emph{translation count} of $K$. The translation count is a $\ZZ^n$-periodic function, which is why we can equivalently consider it as a function on the flat torus $\RR^n/\ZZ^n$.

Many functionals on the set of convex bodies can be expressed in a natural way through the translation count. For instance, one has
\begin{equation}\label{eq:volIsIntegral}
	\vol(K) = \int_{\RR^n/\ZZ^n} \tcount_K(x)\,\mathrm{d}x
\end{equation}
and, more generally,
\begin{equation}
\label{eq:tcount_dilation}
|mK\cap \ZZ^n| = \sum_{z\in\{0,\dots,m-1\}^{n}} \tcount_K(\tfrac 1m z),
\end{equation}
for any $m\in\ZZ_{>0}$ (see Equation~\eqref{eq:tmPEigenfct} below). This raises the question if $K$ is uniquely determined by its translation count $\tcount_K$. Here, ``uniquely determined'' has to be understood up to integer translations, since, clearly, $\tcount_{K+z} = \tcount_K$ holds for all $z\in\ZZ^n$.

It was shown by Higashitani, Murai and Yoshinaga that this is the case for rational polytopes \cite{HigashitaniEtAl} (see also \cite{Alhajjar}).
	
\begin{theorem}[Higashitani, Murai, Yoshinaga]
\label{thm:higashitani}
Let $P,Q\subset\RR^n$ be rational polytopes with $\tcount_P(v) = \tcount_Q(v)$ for all $v\in\RR^n$. Then, $P=Q-z$ for some $z\in\ZZ^n$.
\end{theorem}	

In \cite{HigashitaniEtAl}, Higashitani et al.\ state their theorem under the seemingly stronger assumption that
$\tcount_{mP}(v) = \tcount_{mQ}(v)$
holds for all $m\in\ZZ_{\geq 0}$ and $v\in\RR^n$. In view of \eqref{eq:tcount_dilation}, this is equivalent to the formulation above. However, considering real, or rational, dilates changes the picture drastically. Royer considered the question whether a convex body $K\subset\RR^n$ is determined completely (not only up to integer translation) by the function
\begin{equation}
\label{eq:royer}
\RR_{>0} \times \RR^n\to\ZZ_{\geq 0},\quad (\lambda,v)\mapsto |\lambda (K -v)\cap\ZZ^n|. 
\end{equation}
Royer answers his question affirmatively for rational polytopes \cite{Royer_rational} and centrally symmetric convex bodies \cite{Royer_cs}. Recently, Rocha-Neves gave an affirmative answer to Royer's question for arbitrary convex bodies \cite{Rocha}. Another remarkable result in this context was obtained by Robins~\cite{Robins2026} who showed (in particular) that for rational polytopes it suffices to consider homothetic copies of the form $\lambda(P-v_0)$, where $\lambda\in\QQ_{>0}$ and $v_0 = (2^{1/(d+1)}, \dots, 2^{d/(d+1)})$.

Our first main result is the generalization of Theorem \ref{thm:higashitani} to convex bodies.

	\begin{MainTheorem}
		\label{mthm:inj}
		Let $K,L\subset\RR^n$ be convex bodies with $\tcount_K(v) = \tcount_L(v)$ for all $v\in\RR^n$. Then, $K=L-z$ for some $z\in\ZZ^n$.
	\end{MainTheorem}
	
From the perspective of Fourier analysis, it is tempting to ask whether pointwise equality of the translation counts is necessary in Theorem \ref{mthm:inj}, or if the assumption can be weakened to $\tcount_K(v) = \tcount_L(v)$ for almost all $v\in\RR^n$. This is not the case: It is easy to verify that the translation counts of the rectangles $K = [0,1]\times[0,2]$ and $L= [0,2]\times[0,1]$ agree on $(\RR\setminus\ZZ)^2$, while $K$ cannot be obtained from $L$ via an integer translation. A necessary and sufficient criterion for two  polytopes $P,Q\subset\RR^n$ to satisfy $\tcount_P = \tcount_Q$ almost everywhere is implicit in the article~\cite{LevLiu} by Lev and Liu. It is obtained by observing that $\tcount_P = \tcount_Q$ holds almost everywhere, if and only if $\tcount_P-\tcount_Q$ is a ``tiling at level zero'' and then combining \cite[Theorem 1.2]{LevLiu} and \cite[Proposition 3.3]{LevLiu}. We state the theorem here to complete the picture and refer to \cite{LevLiu} for a definition of the Hadwiger functionals.

\begin{theorem}[Lev, Liu]
\label{thm:levliu}
Let $P,Q\subset\RR^n$ be unions of polytopes. Then, $\tcount_P = \tcount_Q$ holds almost everywhere if and only if the \emph{Hadwiger functionals} $H_\Phi$ agree for $P$ and $Q$ as $\Phi$ ranges over all (oriented) flags of affine subspaces of $\RR^n$.
\end{theorem}

Next, we turn to the question which functions on $\torus{n}$ can be expressed as a translation count. 
In this paper we consider the problem for {lattice polytopes}. We write $\polys(\ZZ^n)$ for the set of lattice polytopes and $\functions(\torus{n})$ for the set of functions on the torus. In the setting of lattice polytopes, the translation count admits a homogeneous decomposition: For any $P\in\polys(\ZZ^n)$ there exist functions $\tcount_{P,i}\in\functions(\torus{n})$, $i=0, \dots, n$, such that for any $m\in\ZZ_{\geq 0}$ we have
\begin{align}\label{eq:homDecompTcount}
		\tcount_{mP} = \sum_{i=0}^{n} m^i\,\tcount_{P,i}.
\end{align}
This formula is a direct consequence of McMullen's homogeneous decomposition theorem for translation invariant valuations on lattice polytopes \cite{McMullen1978}, see also \cite{Alhajjar, HigashitaniEtAl}. We call $\tcount_{P,i}$ the $i$th \emph{Ehrhart translation coefficient}. Indeed, the numbers $\tcount_{P,i}(0)$ are the classical Ehrhart coefficients and some properties of the Ehrhart coefficients generalize: Comparing coefficients it is easy to see that $\tcount_{mP,i} = m^i \tcount_{P,i}$ and, for any $v\in\RR^n$,
\begin{equation}\label{eq:extr_degrees}
\tcount_{P,n}(v) = \vol(P),\quad\tcount_{P,0}(v) = \eins_{\ZZ^n}(v).
\end{equation}
The first equation follows from the definition of the Riemann integral, while the second equation is obtained by setting $m=0$ in \eqref{eq:homDecompTcount}.

In a similar spirit to \eqref{eq:tcount_dilation}, one finds that for any $P\in\polys(\ZZ^n)$, the $i$th Ehrhart translation coefficient $\tcount_{P,i}$ satisfies the functional equations
\begin{equation}\label{eq:funct_eq}
m^i f(v) = \sum_{u\in\torus{n}\,:\, mu=v} f(u),
\end{equation}
for all $m\in\ZZ_{>0}$, see Proposition \ref{prop:tpiEigenfct}. Note that the range of $u$ in the above sum can be represented in $\RR^n$ by the vectors $\tfrac 1m (v+z)$, $z\in\{0,\dots,m-1\}^n$, where we identify $v\in\torus{n}$ with one of its representatives.

It is instructive to interpret the functional equations from \eqref{eq:funct_eq} as a simultaneous eigenfunction condition for the operators
\[
\kubert_m \colon \functions(\torus{n})\to\functions(\torus{n}),\quad \kubert_m f(v) = \tfrac 1{m^n} \sum_{u\colon mu=v} f(u).
\] 
To the best of our knowledge, operators of this form were first considered by Kubert and Lang in the context of toroidal distributions and modular units \cite{Kubert, KubertLang76,KubertLang81}, see also \cite{Lagarias, Milnor}. We refer to $\kubert_m$ as the \emph{Kubert--Lang operators}.

We will see that the functions $\tcount_{P,i}$ generate a simultaneous eigenspace of the Kubert--Lang operators restricted to the vector space of piecewise constant functions on some toric arrangement. In fact, a slightly refined statement holds true. To phrase it, consider a periodic hyperplane arrangement $\hyp$ given by primitive lattice vectors $a_1,\dots,a_N\in(\ZZ^n)^\star$. That is, the hyperplanes of $\hyp$ are $\{x\in\RR^n : \langle a_i,x\rangle = b\}$, where $i\in [N]$, $b\in\ZZ$. After projecting on the flat torus, $\hyp$ is in 1-1 correspondence to a (central) toric arrangement, which we shall also denote by $\hyp$. The set of relatively open cells of the toric arrangement is denoted by $\hypcells$. We write $\pc(\hyp)$ for the vector subspace of $\functions(\torus{n})$ of functions that are constant on the cells in $\hypcells$. To a lattice polytope $P\in\polys(\ZZ^n)$ we associate a toric arrangement $\hyp_P$ via its primitive facet normals and we write $\polys(\ZZ^n;\hyp)$ for the set of lattice polytopes $P$ with $\hyp_P=\hyp$. 
The theory of toric arrangements has further interesting connections to lattice polytopes, in particular to the Ehrhart theory of zonotopes, see, for instance \cite{DadderioMoci, Moci, Procesi}. Further interesting functional equations in the context of Ehrhart theory were studied in \cite{AlfesMaglioneVoll, GunnellsRodriguez}.

It was already observed by Higashitani et al.\  in \cite{HigashitaniEtAl}*{Thm.~1.1} that $\tcount_{P,i}\in\pc(\hyp)$ for any $P\in\polys(\ZZ^n;\hyp)$.
We write $\pc(\torus{n}) = \bigcup_{\hyp}\pc(\hyp)$ for the vector space of piecewise constant functions (with respect to some toric arrangement).
	
	\begin{MainTheorem}
		\label{mthm:eihyp}
		For any $i\in\{0,...,n\}$,
		\begin{equation}
		\label{eq:joint_eigenspace_hyp}
			\spann\{\tcount_{P,i} : P\in\polys(\ZZ^n;\hyp)\} = \bigcap_{m \in \ZZ_{>0}}\left\{f\in\pc(\hyp) : \kubert_{m}f = m^{i-n}f\right\}.
		\end{equation}
		In particular,
		\begin{equation}
		\label{eq:joint_eigenspace}
		\spann\{\tcount_{P,i} : P\in\polys(\ZZ^n)\} = \bigcap_{m \in \ZZ_{>0}}\left\{f\in\pc(\torus{n}) : \kubert_{m}f = m^{i-n}f\right\}.
		\end{equation}
	\end{MainTheorem}	

We denote the simultaneous eigenspaces in Equations \eqref{eq:joint_eigenspace_hyp} and \eqref{eq:joint_eigenspace} by $E_i(\hyp)$ and $E_i(\torus{n})$, respectively.
One inclusion in \eqref{eq:joint_eigenspace} readily follows from \eqref{eq:funct_eq}. The other inclusion will be derived with the help of McMullen's reciprocity theorem. Theorem \ref{mthm:eihyp} is very specific to piecewise constant functions. For more general simultaneous eigenfunctions of the Kubert--Lang operators, see \cite{Milnor,Koblitz}.

Our next main theorem concerns a description of the space $E_{n-1}(\hyp)$ in terms of local data. To state it we first define the \emph{jump} of a function $f\in\pc(\hyp)$ at a point $x\in\torus{n}$ in direction $u\in\RR^n$ by
\begin{equation}
\label{eq:jump_def}
\Delta_uf(x) = \lim_{\varepsilon\searrow 0} f(x) - f(x+\varepsilon u).
\end{equation}
Let us point out that, since $f$ is piecewise constant, the difference is constant for $\varepsilon > 0$ small enough. In particular, the limit always exists.

Moreover, for a primitive lattice vector $a\in(\ZZ^n)^\star$ we write $T_a$ for the $(n-1)$-dimensional torus $\{x\in\torus{n} : \langle a,x\rangle \in \ZZ\}$. From now on it will be beneficial to assume that the description of $\hyp$ in terms of its primitive normal vectors contains the normals in both directions. That is, we describe an arrangement $\hyp$ of $N$ tori by $2N$ primitive lattice vectors $a_1,\dots,a_{2N}\in(\ZZ^n)^\star$ such that for all $i\in [2N]$ there exists a $j\in [2N]$ with $-a_i = a_j$. 

We can now state our main theorem on the structure of $E_{n-1}(\hyp)$.
\begin{MainTheorem}\label{mthm:descrTdm1}
		Let $\hyp$ be a central toric arrangement with primitive normals $a_1,\dots,a_{2N}\in(\ZZ^n)^\star$ (both signs included) and let $f\in E_{n-1}(\hyp)$. Fix $v_i\in\RR^n$ so that $\langle a_i, v_i\rangle > 0$ and let $\delta_i = \int_{T_{a_i}} \Delta_{v_i} f(y)\,\mathrm d_{T_{a_i}} y$, where the integration is with respect to the uniform probability measure on $T_{a_i}$. Then the following linear conditions hold:
		\begin{enumerate}[label=(\Alph*), ref=(\Alph*)]
			\setcounter{enumi}{9}
			\item\label{it:lin2} For every $x\in\torus{n}$ and $v\in\RR^n$ we have
			\begin{align}
				\Delta_v f(x) = \sum_{i : \pair{a_i}{x} \in\ZZ,~ \pair{a_i}{v} > 0} \delta_i.
			\end{align}
			\setcounter{enumi}{1}
			\item\label{it:lin_balance} $\sum_{i=1}^{2N} \delta_i a_i = 0$.
			\setcounter{enumi}{8 }
			\item\label{it:lin3} $\int_{\torus{n}} f(x)\,\mathrm d x = 0$.
		\end{enumerate}
	\end{MainTheorem}
Condition \ref{it:lin2} states, in particular, that the jumps of $f$ are constant on the $(n-1)$-cells of a fixed subtorus. Condition \ref{it:lin_balance} is a balancing condition on the jumps. 
	
In the planar case, we can use \autoref{mthm:descrTdm1} to obtain a description of $\{\tcount_{P,1} : P\in\polys(\ZZ^2;\hyp)\}$ as a finitely generated semigroup inside a vector~space.
\begin{MainTheorem}\label{mthm:descrT1inR2}
		Let $n=2$. If $f \in \pc(\hyp)$ satisfies the conditions \ref{it:lin2}, \ref{it:lin_balance}, and \ref{it:lin3} of \autoref{mthm:descrTdm1}, and, in addition, $\delta_i\in\ZZ_{\geq 0}$,
		then there exists a lattice polygon $P \in \polys(\ZZ^2)$, such that $f = \tcount_{P,1}$.
	\end{MainTheorem}

The fact that $\{ \tcount_{P,1} : P \in \polys(\ZZ^2;\hyp)\}$ is a semigroup is not entirely surprising. Since $P\mapsto \tcount_{P,1}$ is a 1-homogeneous valuation, it is, in particular a semigroup homomorphism with respect to Minkowski addition of polygons, see \cite[Thm.\ 7]{McMullen1978}.

Since the constant and the leading term of \eqref{eq:homDecompTcount} are easily described by \eqref{eq:extr_degrees}, we can provide a description of the functions $\tcount_P$ in the planar case.

\begin{MainCorollary}\label{mcor:tpR2}
		Suppose that $f \in \pc(\hyp)$. Then there exists $P \in \polys(\ZZ^2)$ with $\hyp_P = \hyp$ and such that $f = \tcount_P$ if and only if
		\begin{align*}
			f = \eins_{\ZZ^2} + f_1 + \left(\frac{1}{2}\sum_{1\leq i < j < 2N} \delta_i\delta_j \det(a_i,a_j)\right)\eins,
		\end{align*}
		where $f_1$ satisfies the conditions of \autoref{mthm:descrT1inR2}.
\end{MainCorollary}

Finally, our results for the translation count imply an inequality due to Wills~\cite{Wills1973} as a corollary.

\begin{MainCorollary}[Wills]
	\label{mtc:wills}
	Let $P\subset\RR^2$ be a lattice polygon and let $v\in\RR^2\setminus\ZZ^2$. Then,
	\[
	|(P-v)\cap\ZZ^2| < |P\cap \ZZ^2|.
	\]
\end{MainCorollary}

In general, the problem of finding the extremal values of $\tcount_P$ is very hard, even if $P$ is a lattice polytope. Wills' inequality fails in dimension $n\geq 3$, see \cite{freyerhenk}, where it is conjectured that $\tcount_P(v) \leq 2^{n-2} \tcount_P(0)$ holds for all origin symmetric lattice polytopes. The problem of minimizing $\tcount_P$ is NP-hard, already in the plane \cite{eisenbrand}, as well as for polytopes with a fixed vertex number~\cite{pak}.

Here we obtain a variant of \autoref{mtc:wills} for general dimensions by applying \autoref{mthm:descrTdm1} to $\tcount_{P,n-1}$. It implies Corollary \ref{mtc:wills} by virtue of \eqref{eq:extr_degrees}.
\begin{MainTheorem}\label{mtc:willsFortpnm1}
	Let $P \in \polys(\ZZ^n)$ be a lattice polytope and let $v \in \RR^n \setminus\ZZ^n$. Then
	\begin{align}\label{eq:Willstpnm1}
		\tcount_{P,n-1}(v) \leq \tcount_{P,n-1}(0).
	\end{align}
	Equality holds if and only if $A^T v\in\ZZ^N$, where $A\in\ZZ^{n\times N}$ is the matrix whose columns are the primitive facet normals of $P$.

Moreover, there exists $m_0 \in \ZZ_{>0}$ (depending only on $P$) such that
	\begin{align}
		|(mP-v)\cap\ZZ^n| < |mP\cap \ZZ^n|
	\end{align}
	for all $m \in \ZZ_{\geq m_0}$ and $v\in\RR^n$ for which $A^Tv\not\in\ZZ^N$.
\end{MainTheorem}

Note that there are only finitely many $v\in\RR^n/\ZZ^n$ with $A^Tv\in\ZZ^N$. Indeed, $\Lambda = \{x\in\RR^n : A^Tx\in\ZZ^N\}$ is a superlattice of $\ZZ^n$ and the number of points on the torus $\RR^n/\ZZ^n$ with a preimage in $\Lambda$ is
\[
[\Lambda : \ZZ^n] = \gcd(\det A_I : I\subseteq [N],~ |I|=n),
\]
where $A_I$ is the restriction of $A$ to the columns indexed by $I$.\footnote{The equation seems to be folklore. We refer to \cite[Sec.\ 1.3]{det_ref}, where it is shown that $[\ZZ^n : A\ZZ^n]=\gcd(\det A_I : I\subseteq [N],~ |I|=n)$. It suffices now to observe that $\Lambda$ is the dual lattice of $A\ZZ^n$.}

\medskip

\subsection*{Outline} The paper is organized as follows: In \autoref{sec:prelim} we prove some basic properties of the translation count regarding jumps, boundary structures, and projections. In \autoref{sec:inj}, we prove \autoref{mthm:inj}: we first establish the $2$-dimensional statement and then use an induction argument for general dimension. In \autoref{sec:kubert_lang} we study Kubert--Lang operators and prove \autoref{mthm:eihyp}, \autoref{mthm:descrTdm1}, and \autoref{mtc:willsFortpnm1}. We then turn to the $2$-dimensional case in \autoref{sec:planar}, where we prove \autoref{mthm:descrT1inR2} and \autoref{mcor:tpR2}.

	\section{Preliminaries}\label{sec:prelim}

Recall that the set of lattice polytopes in $\RR^n$ is denoted by $\polys(\ZZ^n)$. Similarly, we write $\polys(\RR^n)$ for the set of general polytopes and $\K(\RR^n)$ for the set of convex bodies.
	
	\subsection{Boundary structures and projections}
	The lemmas in this subsection appear to a large extend in \cite{HigashitaniEtAl} for rational polytopes. Here we slightly adapt their arguments to convex bodies.
	
	We denote the tangent cone of a convex body $K \in \K(\RR^n)$ at a point $x\in K$ by $\tangent(x,K) = \pos\{y-x \colon y\in K\}$. 
	Note that $\tangent(x,K)$ is the dual cone of $\normal(x,K)$, where $\normal(x,K) = \{c\in(\RR^n)^\star \colon \pair{c}{x} =\hf_K(c)\}$ is the normal cone of $K$ at $x$. Here $\hf_K(c) = \max_{x \in K} \pair{c}{x}$, $c \in (\RR^n)^\star$ denotes the support function of $K \in \K(\RR^n)$.
	
	\begin{lemma}
		\label{lemma:tangentcones}
		Let $K\in\K(\RR^n)$ and let $v\in\RR^n$. Then there exists $\varepsilon_0 > 0$ such that for all $\varepsilon\in (0,\varepsilon_0)$ we have
		\[
		(K-\varepsilon v)\cap\ZZ^n = \{z\in K\cap \ZZ^n \colon v\in \tangent(z,K)\}.
		\]
	\end{lemma}
	\begin{figure}
		\begin{tikzpicture}[scale=0.7,>=Stealth] 
	\draw[step=1,gray!25,very thin] (-3,-1) grid (7,6);
	
	\fill[orange!18]
	(-0.35,0) -- (4.65,0) -- (5.65,3) --
	(3.65,5) -- (-1.35,4) -- (-2.35,2) -- cycle;
	
	
	\draw[orange!85!black,thick,dashed]
	(-0.35,0) -- (4.65,0) -- (5.65,3) --
	(3.65,5) -- (-1.35,4) -- (-2.35,2) -- cycle;
	
	\draw[blue!70!black,very thick]
	(0,0) -- (5,0) -- (6,3) --
	(4,5) -- (-1,4) -- (-2,2) -- cycle;
	
	\foreach \y in {2}
	\fill[green!55!black] (-2,\y) circle (2.3pt);
	
	\foreach \y in {1,...,4}
	\fill[green!55!black] (-1,\y) circle (2.3pt);
	
	\foreach \x in {0,...,3}
	\foreach \y in {0,...,4}
	\fill[green!55!black] (\x,\y) circle (2.3pt);
	
	\foreach \y in {0,...,4}
	\fill[green!55!black] (4,\y) circle (2.3pt);
	
	\foreach \y in {1,...,3}
	\fill[green!55!black] (5,\y) circle (2.3pt);
	
	\foreach \x/\y in {5/0,6/3,5/4,4/5}
	\fill[red!75!black] (\x,\y) circle (2.7pt);
	
	\draw[->,very thick] (1.7,2.2) -- (3.0,2.2)
	node[midway,above] {$v$};
	
	\node[blue!70!black] at (1.7,3.5	) {$K$};
	\node[orange!85!black] at (-1.7,0.25) {$K-\varepsilon v$};
	
	
	
	
	\node[align=left,anchor=west] at (0,-1.5) {
		\tikz\fill[red!75!black] (0,0) circle (2.7pt);\
		$\displaystyle z\in K\cap\mathbb Z^2, v\notin\operatorname{Tan}(z,K)$
	};
\end{tikzpicture}
		\caption{Illustration to \autoref{lemma:tangentcones}.}
	\end{figure}
	
	\begin{proof}
		First, we consider
		\[
		\varepsilon_1 = \min\left\{\max\{\alpha>0 \colon z+\alpha v \in K\} \colon z\in K\cap\ZZ^n,~v\in \tangent(z,K)\right\}.
		\]
		For each $z\in K\cap\ZZ^n$ with $v\in\tangent(z,K)$, the number $\max\{\alpha>0 \colon \alpha v\in K-z\}$ is positive by definition of the tangent cone. Since there are only finitely many such lattice points $z$, we have $\varepsilon_1>0$. Moreover, for any $\varepsilon \in (0, \varepsilon_1)$ and any $z\in K\cap\ZZ^n$ with $v\in\tangent(z,K)$ we have $z+\varepsilon v\in K$. Thus, we obtain $\{z\in K\cap \ZZ^n \colon v\in \tangent(z,K)\}\subseteq (K-\varepsilon v)\cap\ZZ^n$ whenever $\varepsilon\in(0,\varepsilon_1)$.
		
		Next, let
		\[
		\varepsilon_2 = \min\left\{ \tfrac{|z-x|}{|v|}\colon z\in \ZZ^n\setminus K,~x\in K\right\},
		\]
		where $|\cdot|$ denotes a norm on $\RR^n$. Since $\ZZ^n$ is discrete and $K$ is compact, we have $\varepsilon_2>0$. Let $\varepsilon\in (0,\varepsilon_2)$ and $z=p-\varepsilon v\in (K-\varepsilon v)\cap\ZZ^n$, $p\in K$. Then $p-z = \varepsilon v$ and by the choice of $\varepsilon_2$, 
		\[
		|z-p| = \varepsilon |v| < \min\{|z-x|\colon z\in\ZZ^n\setminus K,~x\in K\}.
		\]
		In particular, $z \in K$ and $v = \tfrac{1}{\varepsilon} (p-z) \in \tangent(z,K)$. Consequently, $(K-\varepsilon v)\cap\ZZ^n \subseteq \{z\in K\cap \ZZ^n \colon v\in \tangent(z,K)\}$ for all $0<\varepsilon<\varepsilon_2$, and choosing $\varepsilon_0 = \min\{\varepsilon_1, \varepsilon_2\}$ concludes the proof.
	\end{proof}
	
	Note that \autoref{lemma:tangentcones} allows to compute the translation count of a facet $P_c$ of the polytope $P$ almost everywhere using only the jumps of the translation count $t_P$.
	
	Here, a vector $c \in (\RR^n)^\star$ is called an (outer) facet normal of $P \in \polys(\RR^n)$ if and only if the support set
	\begin{align*}
		P_c = \{x\in P \colon \pair{c}{x} = \hf_P(c)\}
	\end{align*}
	is a ($(n-1)$-dimensional) facet of $P$ (an edge, if $n=2$). Support sets are defined similarly for convex bodies $K \in \K(\RR^n)$; we say $K_c$ is a facet if it is of dimension $n-1$.
	
	Recall that the ``jump'' of a piecewise constant function $f\in\pc(\torus{d})$ at a point $x$ in direction $v$ is defined as
	\[
	\Delta_v f (x) = \lim_{\varepsilon\searrow 0} f(x) - f(x+\varepsilon v).
	\]

	\begin{lemma}
		\label{lemma:facet_count}
		Let $P\in\polys(\RR^n)$ be full-dimensional and let $c\in(\RR^n)^\star$ be a facet normal of $P$. Pick any $v \in \RR^n$ with $\pair{c}{v} > 0$.
		
		If $w \in \RR^n$ satisfies
		\begin{align}\label{eq:condLemFacetCount}
			\pair{c'}{w}\not\in \hf_P(c') + c'(\ZZ^n)
		\end{align}
		for all facet normals $c'$ of $P$ that are not multiples of $c$, then
		\[
		\tcount_{P_c}(w) = \Delta_v\tcount_P(w).
		\]
	\end{lemma}
	
	\begin{proof}
		Let $P = \{x\in\RR^n\colon \pair{c_i}{x}\leq \hf_P(c_i),~1\leq i\leq m\}$ for certain $c = c_1,\dots,c_m\in(\RR^n)^\star$ be a facet description of $P \in \polys(\RR^n)$. For simplicity, we assume that $-c=c_m$ defines a facet of $P$ as well. The case where $-c$ is not a facet normal can be treated similarly.
		
		Next, note that a point $x \in P-w$ is contained in the facet $P_{c_i} - w$ if and only if
		\begin{align*}
			\pair{c_i}{x} = \hf_P(c_i) - \pair{c_i}{w}.
		\end{align*}
		Thus, by the assumption on $w$, $(P_{c_i} - w)\cap \ZZ^n = \emptyset$, that is, the facets $P_{c_i}-w$, $1 < i < m$, of $P-w$ do not contain any lattice points. In particular, all lattice points in $\bd(P-w)$ are contained in the relative interior of the two opposite facets $P_{c}-w$ and $P_{-c}-w$.
		
		Moreover, if $x \in \relint(P_c - w)$ then $\tangent(x,P-w) = \{y: \pair{c}{y} \leq 0\}$, and if $x \in \relint(P_{-c} - w)$, then $\tangent(x,P-w) = \{y: \pair{c}{y} \geq 0\}$. Consequently, for the chosen $v \in \RR^n$ with $\pair{c}{v} > 0$, we have that
		\begin{align*}
			\{z\in P\cap \ZZ^n \colon v\in \tangent(z,P)\} = (\interior P \cap \ZZ^n) \sqcup ((P_{-c}-w) \cap \ZZ^n).
		\end{align*}
		By \autoref{lemma:tangentcones}, we can therefore conclude the proof,
		\begin{align*}
			\Delta_v\tcount_P(w) &= \lim\limits_{\epsilon \searrow 0} \left(\tcount_P(w) - \tcount_P(w + \varepsilon v)\right) \\
			&= |(P-w) \cap \ZZ^n| - |\interior P \cap \ZZ^n| - |(P_{-c}-w) \cap \ZZ^n| \\
			&= |(P_{c}-w) \cap \ZZ^n| = \tcount_{P_c}(w).\qedhere
		\end{align*}
	\end{proof}
	\begin{remark}
		As the proof shows, condition~\eqref{eq:condLemFacetCount} asserts that the only lattice points on the boundary of $P-w$ lie in the facets $P_c-w$ and $P_{-c} - w$. Let us note, however, that those facets need not contain a lattice point at all, yielding $\tcount_{P_c}(w) = 0$ in this case.
	\end{remark}
	
	\autoref{lemma:facet_count} can be extended to hold for the translation count of any convex body in $\RR^n$. However, we do not have an explicit condition as in \eqref{eq:condLemFacetCount} to ensure that the remaining boundary does not contain any lattice points.
	
	Here, we denote by $\aff(S)$ the affine hull of a set $S\subset\RR^n$.
	
	\begin{lemma}
		\label{lemma:FacetImpliesJumps}
		Let $K\in\K(\RR^n)$ and let $u\in(\RR^n)^\star\setminus\{0\}$ and $v\in\RR^n$. If $F=K_u$ is a facet, then for almost all $x\in \aff(F)/\ZZ^n$ we have
		\begin{enumerate}
			\item\label{it:FacImpJumpNorm} If $\pair{u}{v} = 0$, then $\Delta_v \tcount_K(x) = 0$, and
			\item\label{it:FacImpJumpTrans} If $\pair{u}{v} > 0$, then $\Delta_v \tcount_K(x) = \tcount_F(x)$. 
		\end{enumerate}
	\end{lemma}
	\begin{proof}
		Let $B = \cl(\bd(K) \setminus (F\cup K_{-u}))$. Then there is only a Lebesgue-null set of points $x\in \aff(F)/\ZZ^n$ for which $(B-x)\cap \ZZ^n \neq \emptyset$. For all other $x\in \aff(F)/\ZZ^n$, we have $(K-x)\cap\ZZ^n = (K-x-\varepsilon v)\cap\ZZ^n$, if $v$ is orthogonal to $u$ and $\varepsilon$ is small enough. Thus, the first statement follows.
		
		To see the second statement, we note that, as before, by excluding all $x \in \aff(F)/\ZZ^n$ from a Lebesgue-null set, we may assume that $(B-x) \cap \ZZ^n  = \emptyset$, that is, $(K-x) \cap \ZZ^n \subseteq \relint(F) \cup \relint(K_{-u})$. Then, by \autoref{lemma:tangentcones}, $\Delta_v \tcount_K(x)$ counts the lattice points $z \in K-x$ for which $v \notin \Tan(z,K-x)$. If $\pair{u}{v} > 0$, this is the case exactly for $z \in F$, as $\Tan(z,K) = \{v\colon \pair{u}{v} \leq 0\}$ for $z \in \relint F$ and $\Tan(z,K) = \{v\colon \pair{u}{v} \geq 0\}$ for $z \in \relint K_{-u}$ and there are no other possibilities by our assumption. Consequently, $\Delta_v \tcount_K(x) = |(F - x) \cap \ZZ^n|$, and the second claim follows.
	\end{proof}
	\begin{remark}
		For $n=2$, the statement of \autoref{lemma:FacetImpliesJumps} holds for all but finitely many $x \in \aff(F)/\ZZ^2$.
	\end{remark}
	
	We will need a reverse statement as well:	
	\begin{lemma}
		\label{lemma:JumpImpliesFacet}
		Let $K\in \K(\RR^n)$ and $x\in\RR^n$. If there exists $u\in(\RR^n)^\star\setminus\{0\}$ such that for any $v\in\RR^n$ we have
		\begin{enumerate}
			\item\label{it:jumpImpFacNorm} If $\pair{u}{v} = 0$, then $\Delta_v \tcount_K(x) = 0$ and
			\item\label{it:jumpImpFacTrans} If $\pair{u}{v} > 0$, then $\Delta_v \tcount_K(x) > 0$, 
		\end{enumerate}
		then, $u$ is the outer normal vector of a facet $K_u$ of $K$. Moreover, $K_u$ contains a point $x'$ such that $x-x' \in \ZZ^n$.
	\end{lemma}
	
	\begin{proof}
		By \autoref{lemma:tangentcones}, the value of $\Delta_v \tcount_K(x)$ is the number of lattice points of $K-x$ that do not contain $v$ in their tangent cone. Hence, it follows from \eqref{it:jumpImpFacNorm} that $u^\perp$ is contained in the tangent cone of all boundary lattice points of $K-x$. Thus, $K-x$ can have boundary lattice points only on $\relint(K_{u}-x)$ and $\relint(K_{-u}-x)$ if either of them is a facet.
		
		Moreover, by \eqref{it:jumpImpFacTrans}, there exists a lattice point $z \in (K-x) \cap \ZZ^n$ and a vector $v \in \RR^n$ for which $v \notin \Tan(z,K-x)$. By the first step, $z$ must be contained either in $K_u - x$ or $K_{-u} - x$, in which case the corresponding support set must be a facet and $z$ contained in the relative interior. However, if $z \in \relint(K_{-u} - x)$, then $\Tan(z,K-x) = \{w\colon \pair{u}{w} \geq 0\} \ni v$, a contradiction. It follows that $z \in \relint(K_u - x) \cap \ZZ^n$, concluding the proof.
	\end{proof}

	\begin{lemma}
		\label{lem:tcConstIffBd}
		Let $K\in\K(\RR^n)$ and let $x\in\RR^n$. Then $\tcount_K$ is constant in a neighborhood of $x$ if and only if $(x + \ZZ^n) \cap \bd K = \emptyset$.
	\end{lemma}
	
	\begin{proof}
		Suppose that $(x + \ZZ^n) \cap \bd K = \emptyset$, that is, $K-x$ contains no boundary lattice points. Then, by compactness of $K$ and discreteness of $\ZZ^n$, $\bd(K-x)$ has a positive distance to $\ZZ^n$, and there exists a neighborhood $U$ of $0\in\RR^n$ such that $(K-x-u)\cap\ZZ^n = (K-x)\cap\ZZ^n$ for all $u\in U$. Thus, $\tcount_K(y) =\tcount_K(x)$ for all $y\in x+U$.
		
		Conversely, suppose that $(x+\ZZ^n) \cap \bd K  \neq \emptyset$, that is, there exists $z \in (\bd K-x) \cap \ZZ^n$. Again, by the compactness of $K$ and the discreteness of $\ZZ^n$, there are no lattice points in a neighborhood of $K-x$ but those which are already contained in $K-x$. That is, there exists an open neighborhood $U'$ of $0 \in \RR^n$ such that $(K-x-u)\cap\ZZ^n \subseteq (K-x)\cap \ZZ^n$ for all $u \in U'$.
		
		Pick $u \in U'$ such that $u \notin \Tan(z,K-x)$. Then, by \autoref{lemma:tangentcones}, $z \notin K-x-u$, and, thus, $(K-x-u)\cap\ZZ^n \subsetneq (K-x)\cap \ZZ^n$. We conclude that $\tcount_K(x+u) < \tcount_K(x)$. As $u \in U'$ can be chosen arbitrarily small, it follows that $\tcount_K$ is not constant in a neighborhood of $x$.
	\end{proof}
	
	In the following, we write $\Lambda_z \subset \RR^n/\spann\{z\}$ for the quotient identifying $z_1, z_2 \in \ZZ^n$ if $z_1 - z_2 = c z$, $c \in \ZZ$, is a multiple of $z$. Geometrically, this corresponds (after choosing an Euclidean structure adapted to $\ZZ^n$) to taking the orthogonal projection onto $z^\perp$. Correspondingly, we write $\Proj_z v$ resp.\ $\Proj_z K$ for the image of $v \in \RR^n$ resp.\ $K \in \K(\RR^n)$ under $\RR^n \to \RR^n/\spann\{z\}$.
	
The next lemma already appeared in \cite{HigashitaniEtAl}. We prove it here to keep the paper self-contained.	
	\begin{lemma}
		\label{lemma:proj}
		Let $K\in\K(\RR^n)$. If $z\in\ZZ^n$ is primitive, then
		\begin{align}\label{eq:lemProjTcount}
			\tcount_{\Proj_z K}^{\Lambda_z}(\Proj_z v) = \sum_{0<s\leq 1}\Delta_z \tcount_K(v-sz), \quad v \in \RR^n.
		\end{align}
	\end{lemma}
	Let us point out that by $\ZZ^n$-periodicity of $\tcount_K$, the right-hand side of \eqref{eq:lemProjTcount} only depends on $\Proj_z v$.
	
	\begin{figure}
		\centering
			\begin{tikzpicture}[
		>=Latex,
		line cap=round,
		line join=round,
		font=\small
		]
		
		\begin{scope}[x=1cm,y=1cm]
			
			\draw[step=1, gray!30, very thin] (-.5,-.5) grid (3.6,3.4);
			
			\foreach \x in {0,1,2,3}
			\foreach \y in {0,1,2,3}
			\fill[black!40] (\x,\y) circle (1.1pt);
			
			\filldraw[
			fill=green!18,
			draw=green!45!black,
			thick
			] (0,0) -- (3,0) -- (3,1) -- (2,3) -- (0,2) -- cycle;
			
			\filldraw[
			fill=blue!17,
			draw=blue!70!black,
			very thick
			] (0,0) -- (3,0) -- (2,2) -- (0,1) -- cycle;
			
			\node[blue!70!black] at (1.35,.55) {$K$};
			\node[green!45!black] at (-1,1.75) {$K+[0,z]$};
			
			\draw[teal!75!black, very thick]
			(3,0) -- (2,2) -- (0,1);
			\node[teal!75!black, anchor=west] at (3.12,.15)
			{\scriptsize $\operatorname{bd}_z(K)$};
			
			\draw[orange!85!black, densely dashed, very thick]
			(3,.5) -- (2,2.5) -- (0,1.5);
			\node[orange!85!black, anchor=west] at (3.12,.65)
			{\scriptsize $\operatorname{bd}_z(K)+\frac12z$};
			
			\draw[red!75!black, densely dashed, very thick]
			(3,1) -- (2,3) -- (0,2);
			\node[red!75!black, anchor=west] at (3.12,1.15)
			{\scriptsize $\operatorname{bd}_z(K)+z$};
			
			\foreach \x/\y in {0/2,1/2,2/3,3/1}
			\filldraw[fill=red!80!black, draw=white, line width=.35pt]
			(\x,\y) circle (2.45pt);
			
			\draw[-{Latex[length=2mm]}, very thick]
			(-0.55,.10) -- (-0.55,1.10)
			node[midway,left] {$z$};
			
		\end{scope}
		
		\begin{scope}[shift={(0cm,-1cm)}, x=1cm, y=1cm]
			
			\draw[-{Latex[length=2mm]}] (-.5,0) -- (3.8,0)
			node[right] {$\mathbb{R}^2/\operatorname{span}(z)$};
			
			\draw[blue!70!black, line width=2.3pt] (0,0) -- (3,0);
			
			\foreach \x in {0,1,2,3}
			\filldraw[fill=red!80!black, draw=white, line width=.35pt]
			(\x,0) circle (2.45pt);
			
			
			
%
		\end{scope}
		
	\end{tikzpicture}
		\caption{Illustration to \autoref{lemma:proj}.}
		\label{fig:projection-jump-identity}

	\end{figure}
	\begin{proof}
		As $Z$ is assumed to be primitive, there are no lattice points in $(0,z)$. Then, for any $v \in \RR^n$ one has 
		\[
		\tcount_{\Proj_z K}^{\Lambda_z} (\Proj_z v) = |(\Proj_z K-\Proj_z v) \cap \Lambda_z| = |((K-v) + [0,z]) \cap \ZZ^n| - |(K-v) \cap \ZZ^n|.
		\]
		Let $\bd_z(K)$ be the upper boundary of $K$ with respect to $z$, that is, the set of all boundary points of $K$ that have an outer normal vector $u$ with $\pair{u}{z} > 0$, or equivalently, all boundary points $x$ with $z \notin \Tan(K,x)$. Then,
		\[
		|((K-v) + [0,z])\cap \ZZ^n| - |(K-v) \cap \ZZ^n| = \sum_{0<s\leq 1} |(\bd_z(K) -v +sz) \cap \ZZ^n|.
		\]
		Note that only finitely many summands are non-zero on the right-hand side. For any $s\in (0,1]$, by \autoref{lemma:tangentcones}, we have
		\[
		(\bd_z(K) -v +sz) \cap \ZZ^n = \left((K -v +sz) \cap \ZZ^n\right) \setminus \left((K -v +(s-\varepsilon)z) \cap \ZZ^n\right)
		\]
		for all $\varepsilon > 0$ small enough. Consequently,
		\[
		|(\bd_z(K) -v +sz) \cap \ZZ^n| = \tcount_K(v-sz) - \lim\limits_{\varepsilon \searrow 0} \tcount_K(v-sz+\varepsilon z) = \Delta_z \tcount_K(v-sz),
		\]
		and the claim follows.
	\end{proof}
	
	\subsection{Translation counts in $\RR^1$}

	It is an instructive exercise to calculate the translation counts for intervals.
	\begin{proposition}\label{prop:tcInR1}
		Let $a < b \in \RR$. Then, for $\{a\} < \{b\}$,
		\begin{align*}
			\tcount_{[a,b]}(v) = \begin{cases}
				\lfloor b-a \rfloor &, v < \{a\},\\
				\lfloor b-a \rfloor + 1 &, \{a\} \leq v \leq \{b\},\\
				\lfloor b-a \rfloor &, \{b\} < v,
			\end{cases} \quad v \in [0,1),
		\end{align*}
		and for $\{a\} > \{b\}$,
		\begin{align*}
			\tcount_{[a,b]}(v) = \begin{cases}
				\lfloor b-a \rfloor + 1 &, v < \{b\},\\
				\lfloor b-a \rfloor &, \{b\} \leq v \leq \{a\},\\
				\lfloor b-a \rfloor + 1 &, \{a\} < v,
			\end{cases} \quad v \in [0,1),
		\end{align*}
		and for $b-a \in \ZZ_{\geq 0}$,
		\begin{align*}
			\tcount_{[a,b]}(v) = \begin{cases}
				b-a+1 & , v = \{a\},\\
				b-a   & , v \neq \{a\}, 
			\end{cases} \quad v \in [0,1).
		\end{align*}
		
		In particular, if $P$ is a lattice polytope in $\RR^1$, then $t_P(v) = \vol_1(P)$ for all $v \in \RR\setminus\ZZ$.
	\end{proposition}
	
	\begin{remark}\label{rem:injInR1}
		It is direct from \autoref{prop:tcInR1} that if $\tcount_K = \tcount_M$ for $K, M \in \K(\RR)$, then $K = M + z$ for some $z \in \ZZ$.
	\end{remark}

	\section{Uniqueness}\label{sec:inj}
	In this section, we prove \autoref{mthm:inj}, first in the planar case, then by induction for all dimensions.
	\subsection{The planar case}	
	The goal of this section is to prove the uniqueness in the plane.
	
	\begin{proposition}
		\label{prop:planar}
		Let $K,M\in\K(\RR^2)$. If $\tcount_K=\tcount_M$, then there exists $z\in\ZZ^2$ such that $K = M-z$.
	\end{proposition}
	
	We require two technical lemmas concerning the ``strictly convex arcs'' in the boundary curve of $K$. We denote by $\pi: \RR^2 \to \torus{2}$ the canonical projection and by $\pi_z = \pi|_{z + [0,1)^2}$, $z \in \ZZ^2$ its restriction to a fundamental cell $z+[0,1)^2$. Clearly, $\pi_z$ is a continuous bijection. 
	
	\begin{lemma}
		\label{lemma:strict1}
		Let $K\in\K(\RR^2)$ and let $R\subset\bd K$ be a regular strictly convex curve. Then, the set of points in $\pi(R) \subset \torus{2}$ that have more than one preimage in $\bd K$ is contained in
		\begin{align}\label{eq:StrictCurvePreim}
			\bigcup_{z,z'} \bigg[(z+[0,1)^2) \cap R-z\bigg] \cap \bigg[(z'+[0,1)^2) \cap \bd K -z'\bigg],
		\end{align}
		where the union ranges over all pairs of distinct integer vectors $z \neq z' \in \ZZ^2$. In particular, only finitely many points in $\pi(R)$ have a non-unique preimage in $\bd K$.
	\end{lemma}
	
	\begin{figure}
		\definecolor{Rred}{RGB}{196,55,55}
\definecolor{Bblue}{RGB}{45,105,180}
\definecolor{Kgray}{RGB}{232,234,238}
\begin{tikzpicture}[
	>={Latex},
	every node/.style={font=\small},
	redcurve/.style={
		draw=Rred,
		line width=2.5pt,
		line cap=round
	},
	bluecurve/.style={
		draw=Bblue,
		line width=2.1pt,
		line cap=round
	},
	boundary/.style={
		draw=black!80,
		line width=.9pt
	},
	point/.style={
		circle,
		fill=black,
		draw=white,
		line width=.4pt,
		inner sep=1.7pt
	}
	]
	
	\pgfmathsetmacro{\rad}{sqrt(0.5)}
	
	\begin{scope}[x=2.25cm,y=2.25cm]
		
		\coordinate (C) at (1.2,1.1);
		\coordinate (p) at (.7,1.6);
		\coordinate (q) at (1.7,1.6);
		
		\fill[fill=Rred!9]  (0,1) rectangle (1,2);
		\fill[fill=Bblue!9] (1,1) rectangle (2,2);
		
		\fill[fill=Kgray,fill opacity=.82]
		(C) circle[radius=\rad];
		
		\draw[step=1,draw=gray!50,line width=.35pt]
		(0,0) grid (2,2);
		
		\draw[draw=Rred!45,line width=1pt]
		(0,1) rectangle (1,2);
		\draw[draw=Bblue!45,line width=1pt]
		(1,1) rectangle (2,2);
		
		\draw[-{Latex[length=1.8mm]},gray!75]
		(-.06,0) -- (2.16,0) node[right] {$x_1$};
		\draw[-{Latex[length=1.8mm]},gray!75]
		(0,-.06) -- (0,2.16) node[above] {$x_2$};
		
		\draw[boundary] (C) circle[radius=\rad];
		
		\begin{scope}
			\clip (1,1) rectangle (2,2);
			\draw[bluecurve] (C) circle[radius=\rad];
		\end{scope}
		
		\draw[redcurve]
		($(C)+(118:\rad)$)
		arc[start angle=118,end angle=153,radius=\rad];
		
		
		\node[point] at (p) {};
		\node[anchor=south east] at (p) {$x$};
		
		\node[point] at (q) {};
		\node[anchor=south west] at (q) {$x'$};
		
		
%
		
		\node at (1.2,.91) {$K$};
		
		
		\draw[
		draw=Rred,
		-{Latex[length=1.4mm]},
		thin
		]
		(.40,1.72) node[left,text=Rred] {$R$}
		-- ($(C)+(140:\rad)$);
		
	\end{scope}
	
	\begin{scope}[shift={(7.05cm,.8cm)},x=3.65cm,y=3.65cm]
		
		\coordinate (Cr) at (1.2,.1); 
		\coordinate (Cb) at (.2,.1);  
		\coordinate (u)  at (.7,.6);
		
		\fill[gray!3] (0,0) rectangle (1,1);
		
		\begin{scope}
			\clip (0,0) rectangle (1,1);
			
			\draw[bluecurve] (Cb) circle[radius=\rad];
			
			\draw[redcurve]
			($(Cr)+(118:\rad)$)
			arc[start angle=118,end angle=153,radius=\rad];
		\end{scope}
		
		\draw[very thick]
		(0,1) -- (0,0) -- (1,0);
		\draw[very thick,densely dashed]
		(1,0) -- (1,1) -- (0,1);
		
		
		\node[point] at (u) {};
		\node[
		anchor=west,
		inner sep=1pt
		] at (.75,.6) {$y$};
		
%
		
		\node[anchor=north east] at (0,0) {$0$};
		\node[anchor=north]      at (1,0) {$1$};
		\node[anchor=east]       at (0,1) {$1$};
		
		
		
	\end{scope}
	
	
	\draw[draw=Rred,thick,-{Latex[length=2mm]}]
	(4.95,3.48) -- (6.85,3.48)
	node[midway,above,text=Rred,fill=white,inner sep=1pt] {$-z$};
	
	\draw[draw=Bblue,thick,-{Latex[length=2mm]}]
	(4.95,2.50) -- (6.85,2.50)
	node[midway,below,text=Bblue,fill=white,inner sep=1pt] {$-z'$};
	
	
	
\end{tikzpicture}
		\caption{Illustration to \autoref{lemma:strict1}.}
	\end{figure}

	\begin{proof}
		First note that cells of the form $(z + [0,1)^2)$, $z \in \ZZ^2$, form a tessellation of $\RR^2$ and thereby their intersections with $\bd K$ cover the whole boundary $\bd K$. For the inclusion statement, we therefore only need to show that two distinct preimages $x, x' \in \bd K$ of a point $y \in \pi(R)$ lie in different cells. This is clear, however.
		
		For the second statement, note that, by the compactness of $K$, we have $(z+[0,1)^2) \cap R \neq \emptyset$ and $(z'+[0,1)^2) \cap \bd K \neq \emptyset$ only for finitely many pairs $z,z' \in \ZZ^2$. Moreover, if $z \neq z'$ is a pair that contributes non-trivially to the union in \eqref{eq:StrictCurvePreim}, then both sets are translations of disjoint subcurves of $\bd K$. As $(z+[0,1)^2) \cap R$ is a regular and strictly convex curve, the intersection with a non-zero translation of any other disjoint curve in the boundary must be either transversal or tangent to the start or end point of the curve. In both cases, the intersection is finite, yielding the second claim.
	\end{proof}

	\begin{lemma}
		\label{lemma:strict2}
		Let $K\in \K(\RR^2)$ and let $U\subseteq \torus{2}$ be open, such that $\rho = U \cap \pi(\bd K)$ is a connected, regular and strictly convex curve. Then there exists a partition of $\rho$ into regular strictly convex curves $\rho_1, \dots, \rho_m$ such that $\pi_{z_i}^{-1}(\rho_i)\subseteq \bd K$ for certain vectors $z_i \in \ZZ^2$, $i=1, \dots, m$.
	\end{lemma}
	\begin{figure}
			\begin{tikzpicture}[
		every node/.style={font=\small},
		lift/.style={line width=1.6pt,line cap=round}
		]
		
		
		\fill[blue!8] (2.10,1.60) ellipse [x radius=1.65,y radius=1.12];
		
		\foreach \x in {0,1,2,3,4}
		\draw[gray!45,densely dashed] (\x,0) -- (\x,3);
		\foreach \y in {0,1,2,3}
		\draw[gray!45,densely dashed] (0,\y) -- (4,\y);
		
		\draw[-{Latex[length=1.8mm]},gray!75]
		(-.06,0) -- (4.66,0) node[right] {$x_1$};
		\draw[-{Latex[length=1.8mm]},gray!75]
		(0,-.06) -- (0,3.66) node[above] {$x_2$};
		
		\draw[black!70,thick] (2.10,1.60) ellipse [x radius=1.65,y radius=1.12];
		\node[black!70] at (3.6,0.5) {$\partial K$};
		
		\coordinate (p0) at (0.70,2.193);
		\coordinate (p1) at (1.00,2.435);
		\coordinate (p2) at (2.00,2.718);
		\coordinate (p3) at (3.00,2.539);
		\coordinate (p4) at (3.641,2.000);
		
		\draw[lift,red!80!black]
		(p0) arc[start angle=148.05,end angle=131.81,
		x radius=1.65,y radius=1.12];
		
		\draw[lift,orange!90!black]
		(p1) arc[start angle=131.81,end angle=93.47,
		x radius=1.65,y radius=1.12];
		
		\draw[lift,teal!75!black]
		(p2) arc[start angle=93.47,end angle=56.94,
		x radius=1.65,y radius=1.12];
		
		\draw[lift,violet!85!black]
		(p3) arc[start angle=56.94,end angle=20.92,
		x radius=1.65,y radius=1.12];
		
		\foreach \p in {p1,p2,p3}
		\fill[black] (\p) circle[radius=1.05pt];
		
		
		
		\draw[-{Latex[length=3mm]},very thick] (4.35,1.62) -- (6.50,1.62)
		node[midway,above] {$\pi$};

		\begin{scope}
			\clip (7,0) rectangle (11,4);
			
			\foreach \i in {0,...,3}{
				\foreach \j in {0,...,2}{
					\draw[gray!70,line width=.75pt]
					({7+4*(2.10-\i)},{4*(1.60-\j)})
					ellipse [x radius=6.6,y radius=4.48];
				}
			}	
			
%
%
%

			\draw[lift,red!80!black]
			({15.4+6.6*cos(148.05)},{-1.6+4.48*sin(148.05)})
			arc[start angle=148.05,end angle=131.81,
			x radius=6.6,y radius=4.48];
			
			\draw[lift,orange!90!black]
			({11.4+6.6*cos(131.81)},{-1.6+4.48*sin(131.81)})
			arc[start angle=131.81,end angle=93.47,
			x radius=6.6,y radius=4.48];
			
			\draw[lift,teal!75!black]
			({7.4+6.6*cos(93.47)},{-1.6+4.48*sin(93.47)})
			arc[start angle=93.47,end angle=56.94,
			x radius=6.6,y radius=4.48];
			
			\draw[lift,violet!85!black]
			({3.4+6.6*cos(56.94)},{-1.6+4.48*sin(56.94)})
			arc[start angle=56.94,end angle=20.93,
			x radius=6.6,y radius=4.48];
		\end{scope}
		
		\node[gray!70,fill=white,inner sep=1pt] at (7.85,.38)
		{$\pi(\partial K)$};
		
		\draw[very thick]
		(7,4) -- (7,0) -- (11,0);
		\draw[very thick,densely dashed]
		(11,0) -- (11,4) -- (7,4);
		
		\node[anchor=north east] at (7,0) {$0$};
		\node[anchor=north]      at (11,0) {$1$};
		\node[anchor=east]       at (7,4) {$1$};
		
		
		\node[red!80!black] at (10.05,1.3) {$\rho_1$};
		\node[orange!90!black] at (8.55,2.2) {$\rho_2$};
		\node[teal!75!black] at (7.65,3.30) {$\rho_3$};
		\node[violet!85!black] at (8.2,.82) {$\rho_4$};
		
		
		
	\end{tikzpicture}
		\caption{Illustration to \autoref{lemma:strict2}.}
	\end{figure}
	\begin{proof}
		First note that, by compactness, there exist only finitely many lattice points $z_1, \dots, z_m \in \ZZ^2$ for which $\pi_{z_i}^{-1}(\rho) \cap \bd K \neq \emptyset$, $i=1, \dots, m$. We write $\Gamma_i = (z_i + [0,1)^2) \cap \bd K$ and $\gamma_i = \pi_{z_i}^{-1}(\rho) \cap \bd K \subseteq \Gamma_i$. Clearly, $\rho$ is the union of all $\rho_i = \pi(\gamma_i)$, $i=1, \dots, m$.
		
		Suppose that $\rho_i \cap \rho_j \neq \emptyset$ and choose a point $x \in \rho$ in the intersection. Since $\rho$ is regular, $\rho$ has a unique unit tangent vector $u$ at $x$, which is then also a unit tangent vector of $\gamma_i$ and $\gamma_j$ at $\pi_{z_i}^{-1}(x)$ and $\pi_{z_j}^{-1}(x) \in \bd K$, respectively. consequently, $\bd K$ contains an edge $E \subset \bd K$ parallel to $u$ containing the distinct points $\pi_{z_i}^{-1}(x)$ and $\pi_{z_j}^{-1}(x)$. As $\rho$ is strictly convex and $\pi(E)$ is a line segment in $\pi(\bd K)$ intersecting $U$, this contradicts the assumption that $\rho = U \cap \pi(\bd K)$. Hence, the sets $\rho_i \subseteq \rho$, $i=1, \dots, m$, are disjoint and therefore the claimed partition. 
	\end{proof}
	
	Before proving \autoref{prop:planar}, we recall Minkowski's existence theorem for the surface area measure. For the definition of surface area measure and more details on the theorem, we refer to \cite{Schneider2014}*{Ch.~8}.
	\begin{theorem}[\cite{Schneider2014}*{Thm.~8.2.2}]\label{thm:Minkowski}
		Suppose that $\mu$ is a positive, finite Borel measure on the unit sphere $S^{n-1}$ with center of mass at the origin,
		\begin{align*}
			\int_{S^{n-1}} u\,\mathrm{d}\mu(u) = 0,
		\end{align*}
		and such that $\mu$ is not concentrated on any great subsphere. Then there exists a full-dimensional convex body $K$ for which $S(K,\cdot) = \mu$ is the surface area measure. The body $K$ is unique up to translations.
	\end{theorem}
	\begin{remark}
		If $\mu = \sum_i f_i \delta_{u_i}$ is a sum of Dirac measures $\delta_{u_i}$, $u_i \in S^{n-1}$, then the resulting body $K$ is indeed a polytope and the conditions reduce to
		\begin{align*}
			\sum_i f_i u_i = 0,
		\end{align*}
		and $\spann\{u_i\} = \RR^n$. In the planar case, it is trivial to retrieve $K$ from $\mu$: One simply takes the $\tfrac \pi 2$-rotation of the polygon with edge vectors $f_i u_i$. The above equation ensures that these edges can form a (convex) polygon.
	\end{remark}
	
	We are now ready to prove \autoref{prop:planar}.
	
	\begin{proof}[Proof of \autoref{prop:planar}]
		We will proceed in three steps. First, we will use \autoref{lemma:FacetImpliesJumps} and \autoref{lemma:JumpImpliesFacet} to find a one-to-one correspondence between edges of $K$ and edges of $M$. Then, we will find a similar correspondence for the strictly convex parts of $\bd K$ and $\bd M$, using \autoref{lemma:strict1} and \autoref{lemma:strict2}. In the last step, finally, we conclude that $K$ and $M$ must have equal surface area measures. The claim then follows from the uniqueness up to translations of convex bodies from their surface area measures (see \autoref{thm:Minkowski}).
		
		\medskip
		
		For the first step, let $u$ be the normal vector of an edge $K_u$ of $K$. Then, all but finitely many $x \in \aff(K_u)/\ZZ^2$ satisfy the conclusions \eqref{it:FacImpJumpNorm} and \eqref{it:FacImpJumpTrans} of \autoref{lemma:FacetImpliesJumps}. If we additionally choose $x \in K_u$, we have $0 \in K_u - x$ and thus $\tcount_{K_u}(x) > 0$. Hence, properties \eqref{it:jumpImpFacNorm} and \eqref{it:jumpImpFacTrans} in \autoref{lemma:JumpImpliesFacet} are satisfied for $\tcount_K$, $u$ and $x$, and therefore, by equality, also for $\tcount_M$, $u$ and $x$. It follows that $M$ has an edge $M_u$ with outer normal $u$. Exchanging the roles of $K$ and $M$ we see that $u$ is an outer normal of an edge of $K$, if and only if $u$ is an outer normal of an edge of $M$.
		
		Let $u$ be an outer normal of edges $E = K_u$ of $K$ and $E' = M_u$ of $M$. By \autoref{lemma:FacetImpliesJumps}, the preimages of all but finitely many points $x\in\pi(E)$ satisfy properties \eqref{it:jumpImpFacNorm} and \eqref{it:jumpImpFacTrans} in \autoref{lemma:JumpImpliesFacet} for $K$. Conversely, by \autoref{lemma:JumpImpliesFacet}, any point $x\in \torus{2}$ with a preimage that satisfies \eqref{it:jumpImpFacNorm} and \eqref{it:jumpImpFacTrans} for $M$ is a point in $\pi(E')$. Since \eqref{it:jumpImpFacNorm} and \eqref{it:jumpImpFacTrans} depend solely on the translation count $\tcount_K=\tcount_M$ we obtain $\pi(E) = \pi(E')$, that is, $\pi(K_u) = \pi(M_u)$. Here, we used that $\pi(K_u)$ and $\pi(M_u)$ are line segments in $\torus{2}$ and therefore already determined by all but finitely many points.
		
		Next, we want to show that $K_u$ and $M_u$ have the same length. To this end, we distinguish between irrational and rational directions. If there is no rational vector in $\spann\{u\}$, that is, $u$ is not the multiple of a rational vector, then $\pi|_{K_u}$ is injective: Indeed, suppose that there exist distinct $x,x'\in K_u$ with $\pi(x)=\pi(x')$. Then, $x-x' \in u^\perp$ is a non-zero integer vector, a contradiction. Thus $\pi(K_u)$ is a curve on the torus $\torus{2}$ whose length is precisely the length of $K_u$. Since $\pi(K_u)=\pi(M_u)$, $K_u$ and $M_u$ have the same length in this case.
		
		Next, we assume that $\spann\{u\}$ contains a rational vector. Then, $\Lambda=\ZZ^2\cap u^\perp$ is a 1-dimensional lattice. Let $x_0\in K_u$ be fixed. Then, for all but finitely many $y\in u^\perp/\Lambda$ we have by \autoref{lemma:FacetImpliesJumps}
		\[
		\tcount^\Lambda_{K_u-x_0}(y)=\tcount_{K_u}(x_0 +y) = \Delta_v\tcount_K(x_0+y),
		\]
		where $v$ is any vector with $\langle u,v\rangle>0$.
		Since $\pi(K_u)=\pi(M_u)$ there exists a vector $z\in\ZZ^2$ such that $x_0+z\in M_u$ and again by \autoref{lemma:FacetImpliesJumps} we have
		\[\begin{split}
			\tcount_{M_u-x_0-z}^\Lambda(y)&=\tcount_{M_u}(x_0 +z+y)\\
			& = \Delta_v\tcount_M(x_0+z+y)
			= \tcount_{K_u-x_0}^\Lambda(y)
			.
		\end{split}\]
		where the last equation follows from $\tcount_K=\tcount_M$ and the fact that $\tcount_K$ is $\ZZ^2$-periodic. Equation~\eqref{eq:volIsIntegral} then implies that the lengths of $K_u$ and $M_u$ are equal.
		
		\medskip
		
		For the second step, we turn to the strictly convex parts of the boundary. In view of \autoref{lem:tcConstIffBd} we have $\pi(\bd(K)) = \pi(\bd(M)) =: \bdproj$. Let $R\subseteq \bd K$ be a regular and strictly convex curve. By \autoref{lemma:strict1}, all but finitely many points in $\pi(R)$ have a unique preimage in $\bd K$ under $\pi$.  Therefore, there exist finitely many regular strictly convex curves $\rho_1,\dots,\rho_m$ such that $\pi(R) = \rho_1\cup\cdots\cup\rho_m$, the union is disjoint and each $\rho_i$ can be represented as $\rho_i = \bdproj\cap U_i$ for certain open sets $U_i\subseteq\torus{2}$.
		
		The regular boundary of $K$ can be decomposed into at most countably many disjoint regular curves $(R_i)_{i\in I}$. By the previous argument, we can now refine this decomposition to an at most countable decomposition $(\gamma_j)_{j\in J}$, missing at most countably many points, such that each $\gamma_j$ projects injectively to some $\rho_j = \bdproj \cap U_j$, $U_j \subseteq \torus{2}$ open. We conclude that $\bdproj$ decomposes (up to countably many points) into a family of regular strictly convex curves in $\torus{2}$ and a countable family of straight lines (corresponding to edges of $K$, and thus also of $M$).
		
		Fix one of the regular strictly convex curves $\rho_j \subset \bdproj$. By \autoref{lemma:strict2}, there exists a partition $\rho_j = \rho_j^1 \cup \dots \cup \rho_j^{m_j}$ of $\rho_j$ into connected regular strictly convex curves with $\pi_{z_j^k}^{-1}(\rho_j^k) \subseteq \bd K$ for certain $z_{j}^k\in\ZZ^2$, $1 \leq k \leq m_j$. Applying \autoref{lemma:strict2} to each $\rho_j^k$ for the body $M$, we may assume (possibly after renaming) that $\pi_{w_j^k}^{-1}(\rho_j^k) \subseteq \bd M$ as well for certain $w_{j}^k\in\ZZ^2$, $1 \leq k \leq m_j$. Consequently, each $\rho_j^k$ corresponds to unique parts of the boundary of $K$ and of the boundary of $M$. Those boundary parts are translates of each other by $z_j^k - w_j^k \in \ZZ^2$.
		
		In total, we have obtained partitions $(S_a)_{a \in A}$ and $(S_a')_{a \in A}$ of $\bd K$ and $\bd M$, respectively, where $S_{a}' = S_a + z_a$ for some $z_a \in \RR^2$.		
		Hence, for any set $\omega\subseteq S_1$ we have, writing $\tau(K,\omega)$ for the set of regular boundary points of $K$ with outer unit normal in $\omega$:
		\[\begin{split}
			\surf(K,\omega) = \haus^1(\tau(K,\omega))& = \sum_{a\in A}\haus^1(\tau(K,\omega)\cap S_a)\\
			& = \sum_{a\in A}\haus^1(\tau(M,\omega)\cap (S_a+z_a)) = \surf(M,\omega).
		\end{split}\]
		Therefore, it follows from the uniqueness part in \autoref{thm:Minkowski} that $M = K+v$ for some vector $v\in\RR^2$. It remains to show that $v$ is an integer vector. For this, we can proceed as in \cite{HigashitaniEtAl}: Since the statement of the proposition trivially holds in dimension 1 (see \autoref{rem:injInR1}), we can use \autoref{lemma:proj} to obtain \[
		\Proj_{e_i}K +\Proj_{e_i}v =  \Proj_{e_i} M = \Proj_{e_i} K + z_i\] for $1\leq i\leq 2$ and $z_i\in\ZZ$. This implies $v\in\ZZ^2$ and the proof is complete. 
	\end{proof}
	
	\subsection{The induction step}
	The following lemma will enable the induction step.
	\begin{lemma}
		\label{lemma:linear}
		Let $f\colon S^{n-1}\to\RR$, where $n\geq 3$, be a continuous function and suppose that $f|_{z^\perp}$ is linear for every $z \in (\ZZ^n)^\star \setminus \{0\}$, that is, for every $z$ there exists a $y_z \in (\RR^n)^\star$ such that $f(u) = \pair{y}{u}$, $u\in z^\perp$. Then $f$ is the restriction of a linear function.
	\end{lemma}
	
	\begin{proof}
		First note that the collection of hyperplanes $z^\perp$ with $z \in (\ZZ^n)^\star\setminus\{0\}$ forms a dense subset of the grassmanian of all hyperplanes in $\RR^n$. Hence, we may conclude by the continuity of $f$ that $f$ is linear when restricted to \emph{any} hyperplane. A short linear algebra argument then shows the claim (see, e.g., \cite{Brauner2024b}*{Lem.~3.4}).
	\end{proof}

	\begin{proof}[Proof of \autoref{mthm:inj}]
		The statement is clearly true for $n=1$ (see also \autoref{prop:tcInR1}), and by \autoref{prop:planar} it is true for $n=2$, so let $n\geq 3$. By \autoref{lemma:proj} we have that $\tcount_{\Proj_z K}^{\Lambda_z} = \tcount_{\Proj_z M}^{\Lambda_z}$ holds for all $z\in\ZZ^n\setminus\{0\}$. By induction, we obtain that $\Proj_z K = \Proj_z M + v_z$ for a certain $v_z\in \Lambda_z$. Since $\hf_K$ and $\hf_{\Proj_z K}$ agree on $z^\perp$, we have that $(\hf_K -\hf_M)(u) = \langle u,v_z\rangle$ holds for all $u\in z^\perp$. By \autoref{lemma:linear} we obtain that $\hf_K-\hf_M$ is a linear function, say $(\hf_K - \hf_M)(u) = \pair{u}{v}$ with $v \in \RR^n$. This implies that $K = M+v$ and we have $\Pr_z v = v_z$ for all $z\in\ZZ^n\setminus\{0\}$. If one considers, for instance, the standard basis vectors $e_1,\dots,e_n\in\ZZ^n$ one finds that $v\in\ZZ^n$ as desired. 
	\end{proof}
	
	\section{Kubert--Lang operators}\label{sec:kubert_lang}	
	
	Recall that the Kubert--Lang operator $\kubert_m \colon \functions(\torus{n}) \to \functions(\torus{n})$ is defined for every $m \in \ZZ_{>0}$ by
	\begin{align*}
		\kubert_m f(v) = \tfrac{1}{m^n} \sum_{u \colon mu=v} f(u), \quad f \in \functions(\torus{n}),
	\end{align*}
	or, more explicitly, by
	\begin{align*}
		\kubert_m f(v) = \tfrac 1 {m^n} \sum_{z\in\{0,\dots,m-1\}^n} f\left(\tfrac 1m (v+z)\right).
	\end{align*}
	We further define $\kubert_{m,i} f = m^{n-i} \kubert_{m} f$, $i \in \{0, \dots, n\}$.
	
	In this section, we will study the operators $\kubert_{m}$ and $\kubert_{m,i}$ and, subsequently, prove \autoref{mthm:eihyp} and \autoref{mthm:descrTdm1}.
	\subsection{Auxiliary results}
	First, let us consider some auxiliary results. Recall that $\pi: \RR^n \to \torus{n}$ denotes the canonical projection.
	
	\begin{lemma}
		\label{lemma:riemann}
		Let $f\in\functions(\torus{n})$ be Riemann integrable and let $v\in \torus{n}$. Then,
		\begin{align}\label{eq:limKubertInt}
		\lim_{m\to\infty} \kubert_{m} f(v) = \int_{\torus{n}} f(u)\,\mathrm d u.
		\end{align}
		In particular,
		\begin{enumerate}
			\item $\kubert_m f = f$ holds for all $m \in \ZZ_{>0}$ if and only if $f \equiv c$, $c \in \RR$.
			\item\label{it:lemRiemannInt0} Let $i < n$. Then, $\kubert_{m,i} f = f$ for all $m \in \ZZ_{>0}$ implies
			\begin{align*}
				\int_{\torus{n}} f(x)\,\mathrm d x = 0.
			\end{align*}
		\end{enumerate}
	\end{lemma}
	
	\begin{proof}
		Let $v\in\torus{n}$. Then the sets $\pi(z + [0,\frac{1}{m})^n)$, $z \in\{0,...,m-1\}^n$, form a partition of the torus $\torus{n}$, each cell containing the corresponding vector $\tfrac 1m (v+z)$. Consequently, $\kubert_{m} f(v)$ is a Riemann sum of $f$ with cell width $\frac{1}{m}$, and \eqref{eq:limKubertInt} follows from Riemann integrability of $f$. The two remaining claims are direct consequences.
	\end{proof}
	
	Let us fix now a toric arrangement $\hyp$ and let $\hypcells$ denote the set of relatively open cells $C\subset\torus{n}$ defined by $\hyp$. For a set $X\subseteq\torus{n}$, let $\eins_X\in\functions(\torus{n})$ be its indicator function. By definition, the functions $\eins_C$, $C\in\hypcells$, form a basis of $\pc(\hyp)$.
	
	\begin{lemma}
		\label{lemma:cells}
		Let $C\in\hypcells$ and $m\in\ZZ_{>0}$. Then, $\kubert_{m,0}\eins_C = \tcount_{m\widetilde C}$, where $\widetilde C$ denotes a relatively open cell in the periodic affine arrangement associated to $\hyp$ that projects injectively to $C$.
	\end{lemma}
	
	\begin{proof}
		By definition, we have $\kubert_{m,0}\eins_C(v) = |\{u \in C : mu=v\}|.$ Lifting to $\RR^n$, we obtain
		\[
		\kubert_{m,0}\eins_C(v) = | \widetilde C\cap (\tfrac 1m v + \tfrac 1m \ZZ^n) | = |(m\widetilde C - v) \cap \ZZ^n | = \tcount_{m\widetilde C}(v).\qedhere
		\]
	\end{proof}
	
	Recall that we write $T_a = \pi(a^\perp)$ for the subtorus of $\torus{n}$ defined by a primitive vector $a \in (\ZZ^n)^\star$, and $\Lambda_a^\star = \ZZ^n \cap a^\perp$ for the corresponding sub-lattice. As $a$ is primitive, the hyperplanes $\{x \colon \pair{a}{x} = z\}$, $z \in \ZZ$, are exactly the lattice planes parallel to $a^\perp$. Indeed, by Bezout's identity, we can find a lattice point $p \in \ZZ^n$ with $\pair{a}{p} = 1$, and, hence, every such hyperplane contains lattice points. Consequently, the set of hyperplanes coincides with $\pi^{-1}(T_a)$.
	
	We need a technical lemma to compute $\kubert_m$ on certain indicators of subtori.	
	\begin{lemma}
		\label{lemma:subgroups}
		Let $a\in(\ZZ^n)^\star$ be primitive, $m\in\NN$, and let $v\in\torus{n}$. Then there exists a solution to
		\[
		mu=v,\quad\text{such that}\quad u\in T_a, 
		\]
		if and only if $v\in T_a$.
		If either is the case, then the set of solutions has the form $u_0 + \pi\left(\tfrac 1m \Lambda_a^\star\right)$ for some solution $u_0\in T_a$. 
	\end{lemma}
	
	\begin{proof}
		Let $v\in\torus{n}$ and consider a preimage $\widetilde v\in\RR^n$ under $\pi$. A solution of $mu=v$ has the form $u=\pi \left( \tfrac 1m (\widetilde v + z)\right)$ for some $z\in\ZZ^n$. We have $u\in T_a$ if and only if
		\[
		\langle a, \tfrac 1m (\widetilde v + z )\rangle \in\ZZ,
		\]
		for some $z\in\ZZ^n$. Equivalently,
		\[
		\langle a, \widetilde v \rangle \in -\langle a,z\rangle + m\ZZ,
		\]
		and, hence, $u \in T_a$ implies $\pair{a}{\widetilde{v}} \in \ZZ$, that is $v \in T_a$. On the other hand, if $v \in T_a$, then, as $a$ is primitive, there exists $z \in \ZZ^n$ such that $\pair{a}{\widetilde{v}} + \pair{a}{z} \in m\ZZ$, and thus $u \in T_a$.
		
		For the second part, consider two solutions $u_0,u\in T_a$ of $mu=v$. As we then have $u - u_0\in T_a$ and $m(u-u_0) = 0$, we see that $u - u_0 \in \pi(\tfrac{1}{m}\Lambda_a^\star)$. This proves the claim.
	\end{proof}
	
	\begin{lemma}
		\label{lemma:more_subgroups}
		Suppose that $a_1, \dots, a_k \in (\ZZ^n)^\star$ are primitive, write $X=\bigcap_{i=1}^k T_{a_i}$, and let $C \subseteq X$ be a connected component of $X$ and let $\ell = \dim\spann\{a_1,\dots,a_k\}$.
		
		Then $C$ is an affine $(n-\ell)$-torus and for any  $m \in \ZZ_{>0}$ and $v \in \torus{n}$ we have
		\begin{align*}
			|\{u\in C : mu=v\}| \leq m^{n-\ell}.
		\end{align*}
	\end{lemma}
	
	\begin{proof}
		The preimage of $X$ under the canonical projection is
		\[
		\pi^{-1}(X) = \{x\in\RR^n : A^Tx\in\ZZ^k\} = \bigcup_{\gamma\in\ZZ^k} \mathcal A_\gamma,
		\]
		where $A$ is the matrix whose columns are the vectors $a_i$ and $\mathcal A_\gamma = \{ x\colon A^Tx = \gamma\}$. Note that two of the affine spaces $\mathcal A_\gamma$ and $\mathcal A_{\gamma'}$ project to the same torus, if and only if $\gamma - \gamma' \in A^T\ZZ^n$, so the connected components of $X$ are indexed by the cosets $(A^T\RR^n \cap\ZZ^k)/ A^T\ZZ^n$, of which there are finitely many. The dimension of each component is the dimension of some $\mathcal A_\gamma$ and, thus, equal to $\dim\ker A^T = n-\ell$.
		
		Finally, we fix a component $C=\pi(\mathcal A_\gamma)$ and we consider two solutions $u,u_0\in C$ to $mu=v$. Let $\widetilde u,\widetilde u_0\in \mathcal A_\gamma$ be preimages of $u$, resp., $u_0$. Then, since $m(u-u_0)=0$, we have $\widetilde u - \widetilde u_0 \in \tfrac 1m \ZZ^n$. At the same time, we have $\widetilde u- \widetilde u_0 \in \ker A^T$. Hence, $u-u_0\in \pi\left( \tfrac 1m \ZZ^n \cap \ker A^T\right)$. The claim follows, since the set on the right hand side has cardinality $[ \tfrac 1m \ZZ^n \cap \ker A^T :   \ZZ^n \cap \ker A^T] = m^{n-\ell}$.
	\end{proof}
	\subsection{Characterization of simultaneous eigenspaces}
	In this section, we study the simultaneous eigenspaces of the operators $\kubert_{m}$ on the space $\pc(\hyp)$ of piecewise constant functions on the cells of a central toric arrangement $\hyp$, that is, the sets
	\begin{align*}
		E_i(\hyp) = \bigcap_{m \in \ZZ_{>0}}\left\{f\in\pc(\hyp) : \kubert_{m}f = m^{i-n}f\right\}, \quad i \in \{0, \dots, n\}.
	\end{align*}
	For $i=n$, a complete characterization can be given easily using \autoref{lemma:riemann}:
	\begin{align}\label{eq:charEn}
		 \quad E_n(\hyp) = \{f \in \pc(\hyp)\colon f \equiv c,~ c \in \RR\}.
	\end{align}
	
	To prove \autoref{mthm:eihyp}, we first need to study some properties of the translation count. Let $A$ be an Abelian semigroup. Here, we say that a map $\varphi\colon\polys(\ZZ^n)\to A$ is invariant under lattice translations if $\varphi(P + z) = \varphi(P)$ for every $P \in \polys(\ZZ^n)$ and $z \in \ZZ^n$. We say that $\varphi$ is a vauation, if $\varphi(P\cup Q) + \varphi( P\cap Q) = \varphi(P)+\varphi(Q)$ holds, whenever $P\cup Q,P\cap Q\in\polys(\ZZ^n)$.
The following lemma is immediate.
	
	\begin{lemma}\label{lem:tCountVal}
		The map $\K(\RR^n) \to \functions(\torus{n}),~K \mapsto \tcount_K$  is a valuation that is invariant under translations from $\ZZ^n$.
	\end{lemma}

	Recall from the introduction that, as a consequence of \autoref{lem:tCountVal}, we can apply McMullen's theorem (\cite{McMullen1977}*{Thm.~5}) to obtain the decomposition~\eqref{eq:homDecompTcount} into homogeneous components $\tcount_{P,i}$, $i=0, \dots, n$.
	
	Since $\tcount_{P}$ is piecewise constant with respect to the toric arrangement $\hyp_P$ associated to $P$, it follows from \eqref{eq:homDecompTcount} that $t_{P,i}$ is also piecewise constant with respect to $\hyp_P$.
	
	\begin{proposition}
		\label{prop:tpiEigenfct}
		Let $P\in\polys(\ZZ^n)$ and $i \in \{0, \dots, n\}$. Then $\tcount_{P,i} \in E_i(\hyp_P)$.
	\end{proposition}
	
	\begin{proof}
		Let $m \in \ZZ_{>0}$ and $v \in \torus{n}$. Then
		\begin{align}\label{eq:tmPEigenfct}
			\begin{split}
				\tcount_{mP}(v) &= |(mP-v)\cap \ZZ^n| \\
				&= \left|\left(P - m^{-1}v\right)\cap \left(m^{-1} \ZZ^n\right)\right| = \sum_{z \in \{0,...,m-1\}^n} \tcount_{P}\left(\tfrac 1m(v+z)\right).
			\end{split}
		\end{align}
		Combining \eqref{eq:tmPEigenfct} with the homogeneous decomposition of $\tcount_P$ (see \eqref{eq:homDecompTcount}), we obtain for $\ell \in \ZZ_{> 0}$
		\begin{align*}
				\sum_{i=0}^n (\ell m)^i\, \tcount_{P,i}(v) = \tcount_{\ell m P}(v) 
				&= \sum_{z \in \{0,...,m-1\}^n} \tcount_{\ell P}\left(\tfrac 1m(v+z)\right)\\
				&= \sum_{z \in \{0,...,m-1\}^n}\sum_{i=0}^n \ell^i\, \tcount_{P,i}\left(\tfrac 1m (v+z)\right).
		\end{align*}
		We extract the coefficient of $\ell^i$ on both sides to find
		\[
		m^i\, \tcount_{P,i}(v) = \sum_{u : mu=v} \tcount_{P,i}(u) = m^n\, \kubert_m\tcount_{P,i}(v).\qedhere
		\]
	\end{proof}
	
	Next, we want to express the translation count of the relative interior of a lattice polytope $P \in \polys(\ZZ^n)$ by the homogeneous components of its translation count (resp.\ the translation count of its reflection $-P$). For this, we will need McMullen's reciprocity theorem (\cite{McMullen1977}*{Thm.~12}):
	\begin{theorem}[\cite{McMullen1977}*{Thm.~12}]
		\label{thm:McMrecipr}
		Let $A$ be an Abelian group and suppose that $\varphi : \polys(\ZZ^n) \to A$ is a valuation that is invariant under lattice translations. If $\varphi$ is homogeneous of degree $i \in \{0, \dots, n\}$, then
		\begin{align*}
			\varphi(-P) = (-1)^i \sum_F (-1)^{\dim F} \varphi(F), \quad P \in \polys(\ZZ^n),
		\end{align*}
		where the sum runs over all non-empty faces of $P$.
	\end{theorem}
	Using a standard inclusion-exclusion argument, we deduce:
	\begin{lemma}\label{lem:tCountOnRelint}
		Suppose that $P \in \polys(\ZZ^n)$. Then
		\begin{align*}
			\tcount_{\relint P} = (-1)^{\dim P}\sum_{i=0}^{n}(-1)^{i} \tcount_{-P,i}.
		\end{align*}
	\end{lemma}
	\begin{proof}
		By inclusion-exclusion,
		\begin{align*}
			\tcount_{\relint P}(v) = |(\relint P - v) \cap \ZZ^n| = \sum_F (-1)^{\dim P - \dim F} |(F-v) \cap \ZZ^n|,
		\end{align*}
		where the sum runs over all non-empty faces of $P$ (including $P$ itself). As, by \autoref{lem:tCountVal}, $P \mapsto \tcount_P(v)$ is a valuation that is invariant under lattice translations, \eqref{eq:homDecompTcount} and \autoref{thm:McMrecipr} yield
		\begin{align*}
			\sum_F (-1)^{\dim F} \tcount_{F}(v) = \sum_{i=0}^{n}\sum_F (-1)^{\dim F} \tcount_{F,i}(v) = \sum_{i=0}^{n} (-1)^i \tcount_{-P,i}(v),
		\end{align*}
		and the claim follows.
	\end{proof}
	
	\begin{proof}[Proof of \autoref{mthm:eihyp}]
		In view of \autoref{prop:tpiEigenfct}, the inclusion
		\[
		E_i(\hyp) \supseteq \spann\{\tcount_{P,i} : P\in\polys(\ZZ^n;\hyp)\}
		\]
		is clear. For the converse, we associate a rational polytope $P_C \subset \RR^n$ to every relatively open cell $C\in\hypcells$ such that $\relint(P_C)$ projects injectively to $C$. Let $D\in\NN$ such that $D P_C$ is a lattice polytope for every $C\in\hypcells$ (note that there are only finitely many cells, and that their preimages are rational). Then, $DP_C\in\polys(\ZZ^n;\hyp)$. \autoref{lemma:cells} together with \autoref{lem:tCountOnRelint} yield
		\[
		\kubert_{mD,0} \eins_C = \tcount_{\relint (mDP_C)} = (-1)^{\dim C}\sum_{i=0}^n (-m)^i \tcount_{-DP_C,i}.
		\]
		Now let $f = \sum_{C\in\hypcells} n_C \eins_C\in E_i(\hyp)$. Then, $\kubert_{mD,0}f = (mD)^i f$, for every $m\in \ZZ_{> 0}$. Moreover, with the above equation,
		\[
		(mD)^i f = \sum_{C\in\hypcells} n_C \kubert_{mD,0}\eins_C = \sum_{C\in\hypcells} (-1)^{\dim C} \sum_{i=0}^n (-m)^i n_C \tcount_{-DP_C,i}.
		\]
		Extracting the coefficient of $m^i$, one finds that $f$ is a linear combination of the functions $\tcount_{-DP_C,i}\in\pc(\hyp)$, indexed by $C\in\hypcells$.
	\end{proof}
	
	As a corollary of \autoref{mthm:eihyp}, we immediately obtain (see \eqref{eq:extr_degrees})
	\begin{align}\label{eq:charE0}
		E_0(\hyp) = \spann\{\eins_{\ZZ^n}\}.
	\end{align}
	
	\subsection{Descriptions of $E_{n-1}(\hyp)$}
	
	Next, we want to describe $E_{n-1}(\hyp)$, that is, prove \autoref{mthm:descrTdm1}. Before giving the proof, let us mention the following observation:
	
	\begin{lemma}\label{lem:eigenjumps}
		Let $f \in E_i(\hyp)$ and $v \in \RR^n$. Then $\Delta_v f \in E_i(\hyp)$.
	\end{lemma}
	\begin{proof}
		Let $x \in \torus{n}$ and fix $m \in \ZZ_{>0}$. Then
		\begin{align*}
			\Delta_v f(x) &= \lim_{\varepsilon\searrow 0} (f(x) - f(x + \varepsilon v))\\
			& =  \tfrac 1 {m^{i}} \sum_{z\in \{0,...,m-1\}^n} \lim_{\varepsilon\searrow 0} f(\tfrac 1m (x+z) ) - f(\tfrac 1m (x+\varepsilon v + z))\\
			& = \tfrac 1 {m^{i}} \sum_{z\in \{0,...,m-1\}^n} \Delta_v f(\tfrac{1}{m}(x+z)) = \kubert_{m,i}(\Delta_v f)(x).\qedhere
		\end{align*}
	\end{proof}
	
	We are now ready to prove \autoref{mthm:descrTdm1}. The first proof we give is independent of the valuation theory of lattice polytopes. It works directly with the structure of functions in $\pc(\hyp)$.
	
	\begin{proof}[Proof of \autoref{mthm:descrTdm1}]
		By Lemma \ref{lemma:riemann}\eqref{it:lemRiemannInt0}, we only need to show Items \ref{it:lin2} and \ref{it:lin_balance}. To this end, let $x \in \torus{n}$ and $v \in \RR^n$ be arbitrary, and define
		\begin{align*}
			I = \left\{ i \in \{1, \dots, 2N\}\colon \pair{a_i}{x} \in \ZZ, \pair{a_i}{v} > 0 \right\},
		\end{align*}
		where we recall that $a_1, \dots, a_{2N} \in (\ZZ^n)^\star$ are the (primitive) normal vectors of the central toric arrangement $\hyp$ (both signs included). Formulated differently, $i \in I$ if and only if $x \in T_{a_i}$ and $\pair{a_i}{v}> 0$. Let us further point out that if $a_i = - a_j$, then, by the second condition, not both $i$ and $j$ can be contained in $I$.
		
		Fix $m \in \ZZ_{>0}$. By \autoref{lem:eigenjumps}, $\Delta_v f \in E_{n-1}(\hyp)$ and therefore
		\begin{align}\label{eq:prfTdm1KubJump}
			\Delta_v f(x) = \kubert_{m,n-1}(\Delta_v f)(x) = \tfrac 1 {m^{n-1}} \sum_{x': mx' = x} \Delta_v f(x').
		\end{align}
		Suppose that $x' \in \torus{n}$ is a solution to $mx' = x$. If $x' \in T_{a_i}$ for some $i$, then, by \autoref{lemma:subgroups}, $x \in T_{a_i}$ as well. Hence, if $x'\not\in \bigcup_{i\in I} T_{a_i}$, then $x' \in T_{a_j}$ only for $\pair{a_j}{v} \leq 0$. As $T_{a_j} = T_{-a_j}$, this can happen only if $\pair{a_j}{v} = 0$.
		
		For such $x'$, we have $\Delta_v f(x') = 0$. Indeed, $\pair{a_j}{v} = 0$ for all $j$ where $x' \in T_{a_j}$ implies that $x' + \epsilon v$ is contained in the same cell of $\hypcells$ for $\epsilon > 0$ small enough. As $f\in\pc(\hyp)$ is constant on these cells, $\Delta_v f(x') = 0$. If $x' \notin T_{a_j}$ for any $j \in \{1, \dots, 2N\}$, then $x'$ is contained in the interior of a full-dimensional cell of $\hypcells$. Hence, $\Delta_v f(x') = 0$ also in this case.
		
		Writing $T_J$ for $\bigcap_{j\in J} T_{a_j}$, by inclusion-exclusion of $\bigcup_{i\in I} T_{a_i}$, \eqref{eq:prfTdm1KubJump} therefore becomes
		\begin{align}\label{eq:prfTdm1InclExcl}
		\Delta_v f(x) = \tfrac 1 {m^{n-1}} \sum_{\emptyset \neq J\subseteq I} (-1)^{|J|-1}  \sum_{x' \in T_J\colon mx'=x} \Delta_v f(x').
		\end{align}
		For a set $J\subseteq I$ of cardinality at least $2$, the corresponding vectors $a_j$, $j\in J$, generate a linear space of dimension at least $2$, since $\{a_i\colon i \in I\}$ does not contain pairs $a_i = -a_j$. Hence, by \autoref{lemma:more_subgroups}, $T_J$ decomposes into connected components $C_1,\dots, C_k$ which are isomorphic to tori of dimension at most $n-2$. Thus,
		\[
		\lim_{m\to\infty} \tfrac 1 {m^{n-1}} \sum_{x'\in T_J : mx' = x} \Delta_vf(x') = \sum_{i=1}^k \lim_{m\to\infty} \tfrac 1 {m^{n-1}} \sum_{x'\in C_i : mx' = x}  \Delta_vf(x')= 0,
		\]
		where the last equation follows from the second part of \autoref{lemma:more_subgroups},
		\begin{align*}
			|\{x' \in C_i: mx'=x\}| \leq m^{n-2},
		\end{align*}
		and the fact that $|\Delta_v f|$ is uniformly bounded as $f$ attains only finitely many values on $\torus{n}$.
		Thus, using \autoref{lemma:riemann}, \eqref{eq:prfTdm1InclExcl} implies
		\[
		\Delta_v f(x) = \lim_{m\to\infty} \tfrac 1 {m^{n-1}} \sum_{i\in I} \sum_{x'\in T_{a_i} : mx' = x} \Delta_v f(x') = \sum_{i\in I} \int_{T_{a_i}} \Delta_v f(y)\,\mathrm d_{T_{a_i}} y,
		\]
		which proves Item~\ref{it:lin2} with $\delta_i = \int_{T_{a_i}} \Delta_v f(y)\,\mathrm d_{T_{a_i}} y$.
		
		For Item~\ref{it:lin_balance}, let $z\in\ZZ^n$ be primitive with $\pair{a_i}{z} \neq 0$ for all $i \in \{1, \dots, 2m\}$, and let $v\in\RR^n$ be such that the segment $[v, v+z]$ avoids all cells of the periodic arrangement corresponding to $\hyp$ whose codimension is more than $1$ and that $v$ is contained in a full-dimensional cell. Note that this is a generic property, that is, we can choose almost every $v \in \RR^n$. In particular, such a $v$ exists. Consider the function
		\[
		g \colon [0,1] \to \RR,\quad g(\lambda) = (f\circ \pi)( v + \lambda z)
		\]
		Clearly, we have $g(0) = g(1)$. Whenever the segment $[v,v+z]$ crosses a hyperplane $H$ of the arrangement, the value of $g$ jumps up by $\delta_i$ and then drops down by $\delta_j$, where $a_i$ and $a_j$ are the two primitive normals of $H$ with $\pair{a_i}{z} < 0 < \pair{a_j}{z}$. The number of hyperplanes with normals $a_i$ and $a_j$ that $[v,v+z]$ crosses is $|\pair{a_i}{z}| = |\pair{a_j}{z}|$.
		This yields
		\[
		g(0) = g(1) = g(0) + \sum_{i=1}^{2N} \delta_i \pair{a_i}{z},
		\]
		that is, $\pair{\sum_{i=1}^{2N} \delta_i a_i}{z} = 0$. As we can choose $n$ linearly independent $z \in \ZZ^n$, Item~\ref{it:lin_balance} follows, concluding the proof.
	\end{proof}
	
We sketch another proof of \autoref{mthm:descrTdm1} using \autoref{mthm:eihyp}.	
	
	\begin{proof}[Alternative proof of \autoref{mthm:descrTdm1} using \autoref{mthm:eihyp}]
		As Items~\ref{it:lin2}, \ref{it:lin_balance}, and \ref{it:lin3} are linear in $f$, by \autoref{mthm:eihyp} it is sufficient to prove them for $\tcount_{P,n-1}$, where $P \in \polys(\ZZ^n;\hyp)$.
		
		Denote by $F$ the face of $P$ with outer normal $a_k$, and let $x\in T_{a_k} \setminus \bigcup_{a_j\neq\pm a_k} T_{a_j}$. Then, by \autoref{lemma:facet_count}, we have $\tcount_{mF}(x) = \Delta_{u}\tcount_{mP}(x)$ for any $u \in \RR^n$ with $\pair{a_k}{u} > 0$. Hence, using the homogeneous decomposition (see \eqref{eq:homDecompTcount}), we obtain for any $m\in\NN$
		\begin{align}\label{eq:altProofDescrTdm1}
		\sum_{i=0}^{n-1} m^i \tcount_{F,i}(x) = \tcount_{mF}(x) = \Delta_{u}\tcount_{mP}(x) = \sum_{i=0}^{n-1} m^i \Delta_{u}\tcount_{P,i}(x),
		\end{align}
		where we used that $\tcount_{P,n} \equiv \vol(P)$ (by \eqref{eq:charEn} and \autoref{prop:tpiEigenfct}), and thus $\Delta_{u} \tcount_{P,n}(x) = 0$. Extracting the leading term on both sides gives $\Delta_{u}\tcount_{P,n-1}(x) = \tcount_{F,n-1}(x)$. As $\tcount_F|_{T_{a_k}}$ coincides with $\tcount_F^{\ZZ^n \cap a^\perp}$, the translation count with respect to the lattice in $a^\perp$, the same argument shows that $\tcount_F$ is constant on $T_{a_k}$ (indeed, it is equal to the lattice normalized volume of $F$ in $a^\perp$). Hence $\Delta_{u}\tcount_{P, n-1}$ is constant on the relatively open cells of $\hyp$ in $T_{a_k}$. As the jumps are equal to the facet volumes, \autoref{thm:Minkowski} shows Item~\ref{it:lin_balance}. 
		
		Item~\ref{it:lin2} can be derived with a similar limit argument to the proof above, where one neglects contributions from lower dimensional faces to the jumps of the dilated polytope. Item~\ref{it:lin3} follows again from \autoref{lemma:riemann}.
	\end{proof}
	
	We now turn to the proof of \autoref{mtc:willsFortpnm1} and a sufficient condition for strictness in \eqref{eq:Willstpnm1}.
	
	\begin{proof}[Proof of \autoref{mtc:willsFortpnm1}]
%
		%
		Similar to the proof of \autoref{mthm:descrTdm1}~\ref{it:lin_balance}, by iterating \autoref{mthm:descrTdm1}~\ref{it:lin2}, we obtain for $v \in \RR^n \setminus \ZZ^n$
		\begin{align*}
			\tcount_{P,n-1}(0) - \tcount_{P,n-1}(v) = \sum_{i=1}^{2N}\delta_i \lceil \pair{a_i}{v} \rceil, 
		\end{align*}
		where $a_i$ are the facet normals of $P$, which define $\hyp_P$, and  $\delta_i \in \RR$ are the jumps from \autoref{mthm:descrTdm1}~\ref{it:lin2} (applied to $E_{n-1}(\hyp_P)$), and $\lceil z \rceil = -\lfloor |z|\rfloor$ for negative $z \in \RR$. Note that, by \autoref{lemma:facet_count}, the jumps $\delta_i$ are the facet volumes of $P$:
\[
\delta_i = \int_{T_{a_i}} \Delta_{v_i} \tcount_{P,n-1}(y)\,\mathrm d_{T_{a_i}} y = \int_{T_{a_i}} \tcount_{P_{a_i},n-1}(y)\,\mathrm d_{T_{a_i}} y = \vol(P_{a_i}),
\]
where the volume is normalized by the determinant of the lattice $\ZZ^n\cap a_i^\perp$. In particular, $\delta_i \geq 0$ holds.

		By \autoref{mthm:descrTdm1}~\ref{it:lin_balance}, $\sum_i \delta_i a_i = 0$, and thus
		\begin{align}
			\tcount_{P,n-1}(0) - \tcount_{P,n-1}(v) &= \sum_{i=1}^{2N}\delta_i \lceil \pair{a_i}{v} \rceil - \pair{\sum_{i=1}^{2N} \delta_i a_i}{v}\nonumber\\	
			&=\sum_{i=1}^{2N}\delta_i \left(\lceil \pair{a_i}{v} \rceil - \pair{a_i}{v}\right) \geq 0.\label{eq:prfWillsFtpnm1Ineq}
		\end{align}
		This shows \eqref{eq:Willstpnm1} for all $v \in \RR^n \setminus \ZZ^n$. The characterization of the equality case is also apparent in \eqref{eq:prfWillsFtpnm1Ineq}.
		
		For the second claim note that, by the homogeneous decomposition \eqref{eq:homDecompTcount},
		\begin{align*}
			\tcount_{mP}(0) - \tcount_{mP}(v) = \sum_{i=0}^{n}m^i (\tcount_{P,i}(0) - \tcount_{P,i}(v)),
		\end{align*}
		where the coefficient of $m^n$ vanishes by \eqref{eq:extr_degrees}.	Hence, if the inequality in \eqref{eq:Willstpnm1} is strict, then $\tcount_{mP}(0) - \tcount_{mP}(v)$ is a polynomial of degree $n-1$ in $m$ with positive leading coefficient. Since $\tcount_{mP}$ is piecewise constant with respect to the toric arrangement $\hyp_{mP} = \hyp_P$, we encounter only finitely many polynomials this way. As each such polynomial becomes positive as $m \to \infty$, the second claim follows.
	\end{proof}
	
	\section{Image of the planar translation count}\label{sec:planar}
	
	In this section, we turn to the $2$-dimensional case, proving \autoref{mthm:descrT1inR2} and \autoref{mcor:tpR2}.
	
	\begin{proof}[Proof of \autoref{mthm:descrT1inR2}]
		Assume that $f \in \pc(\hyp)$ satisfies the conditions in \autoref{mthm:descrTdm1} with $\delta_i \in \ZZ_{\geq 0}$ for all $i$. By \ref{it:lin_balance}, the vectors $\tfrac{a_i}{|a_i|}$ and the scalars $|a_i| \delta_i \in \ZZ_{\geq 0}$ satisfy the conditions of \autoref{thm:Minkowski}. Hence, there exists a polygon $P \in \polys(\RR^2)$ with outer unit normals $a_i$ and Euclidean edge lengths $|a_i| \delta_i$. We may assume that $P$ has a vertex at the origin. Since the determinant of the rank-1 lattice $\ZZ^2\cap a_i^\perp$ is $|a_i|$, the lattice normalized length of each edge of $P$ is integral. Translating $P$ so that the origin is a vertex of $P$, we may assume that $P$ is a lattice polygon.
		
		Further, we claim that $f = \tcount_{P,1}$. Indeed, by \autoref{lemma:facet_count} and \autoref{prop:tcInR1}, $\tcount_P$ has jumps $\Delta_u \tcount_P$, $\pair{a_i}{u} > 0$, that are constantly equal to the lattice normalized facet lengths $\tfrac{\vol_1(P_{a_i})}{|a_i|}$ on the relatively open cells of $\hypcells$ parallel to the corresponding subtorus $T_{a_i}$. As $\vol_1(P_{a_i}) = \delta_i |a_i|$ by construction, we see that $\Delta_u \tcount_P = \delta_i = \Delta_u f$ on the relatively open cells. By the characterization of $\tcount_{P,0} = \eins_{\ZZ^2}$ and $\tcount_{P,2} \equiv \vol_2(P)$ (see \eqref{eq:extr_degrees}), we conclude that the jumps of $\tcount_{P,1} = \tcount_P - \eins_{\ZZ^2} - \vol_2(P)\eins$ and $f$ coincide as well: $\Delta_u \tcount_{P,1} = \delta_i = \Delta_u f$. Hence, since both functions have vanishing integral (for $f$ by assumption~\ref{it:lin3}, for $\tcount_{P,1}$ by \autoref{prop:tpiEigenfct} and \autoref{lemma:riemann}), they must coincide, proving the claim.
	\end{proof}
	\begin{remark}
		Since the area of any lattice polygon is in $\tfrac 12 \ZZ$, the proof of \autoref{mthm:descrT1inR2} shows that $f(x) \in \tfrac{1}{2}\ZZ$.
	\end{remark}
	
	\autoref{mcor:tpR2} follows now easily from \autoref{mthm:descrT1inR2}, once we know how to reconstruct the $2$-homogeneous (volume-)term from $\tcount_{P,1}$. As we work only in $\RR^2$, this is a direct calculation:
	
	\begin{lemma}\label{lem:volR2ByFacets}
		Let $P \in \polys(\ZZ^2)$ be a lattice polygon with primitive outer normals $a_i, \dots, a_{2N} \in (\ZZ^2)^\star$ (both signs included). If we denote by $\delta_i \in \ZZ_{\geq 0}$, $i \in \{1, \dots, 2N\}$ the jumps of $\tcount_P$ according to \autoref{mthm:descrTdm1}, then
		\begin{align*}
			\vol_2(P) = \frac{1}{2}\sum_{1\leq i < j < 2N} \delta_i\delta_j \det(a_i,a_j).
		\end{align*}
	\end{lemma}
	Note that $\delta_i |a_i| = \vol_1(P_{a_i})$ here (see the proof of \autoref{mthm:descrT1inR2}).
	\begin{figure}
		\usetikzlibrary{calc}

\begin{tikzpicture}[scale=1.1, line join=round, >=stealth]
	
	\coordinate (v0) at (0,0);
	\coordinate (v1) at (3,0);
	\coordinate (v2) at (4,1);
	\coordinate (v3) at (3,3);
	\coordinate (v4) at (1,4);
	\coordinate (v5) at (-1,2);
	
	
	\draw[black, dashed] (v0)--(v2);
	\draw[black, dashed] (v0)--(v3);
	\draw[black, dashed] (v0)--(v4);
	
	\draw[very thick]
	(v0)--(v1)--(v2)--(v3)--(v4)--(v5)--cycle;
	
	\draw[->, thick, red!75!black]
	($(v0)$) -- ($(v1)$)
	node[midway, below] {$u_1$};
	
	\draw[->, thick, red!75!black]
	($(v1)$) -- ($(v2)$)
	node[midway, right] {$u_2$};
	
	\draw[->, thick, red!75!black]
	($(v2)$) -- ($(v3)$)
	node[midway, right] {$u_3$};
	
	\draw[->, thick, red!75!black]
	($(v3)$) -- ($(v4)$)
	node[midway, above] {$u_4$};
	
	\draw[->, thick, red!75!black]
	($(v4)$) -- ($(v5)$)
	node[midway, left] {$u_5$};
	
	\draw[->, thick, red!75!black]
	($(v5)$) -- ($(v0)$)
	node[midway, left] {$u_6$};
	
	\fill (v0) circle (1.7pt);
	\fill (v1) circle (1.7pt);
	\fill (v2) circle (1.7pt);
	\fill (v3) circle (1.7pt);
	\fill (v4) circle (1.7pt);
	\fill (v5) circle (1.7pt);
	
	\node[below right, font=\small] at (v0) {$v_0=0$};
	\node[below, font=\small]       at (v1) {$v_1=u_1$};
	\node[right, font=\small]       at (v2) {$v_2=u_1+u_2$};
	\node[right, font=\small]       at (v3) {$v_3=u_1+u_2+u_3$};
	\node[above, font=\small]       at (v4) {$v_4=u_1+\cdots+u_4$};
	\node[left, font=\small]        at (v5) {$v_5=u_1+\cdots+u_5$};
	
	\node[font=\large] at (1.6,1.75) {$P$};
	
\end{tikzpicture}
		\caption{Illustration of the triangulation from the proof of \autoref{lem:volR2ByFacets}.}
		\label{fig:volR2byFacets}
	\end{figure}
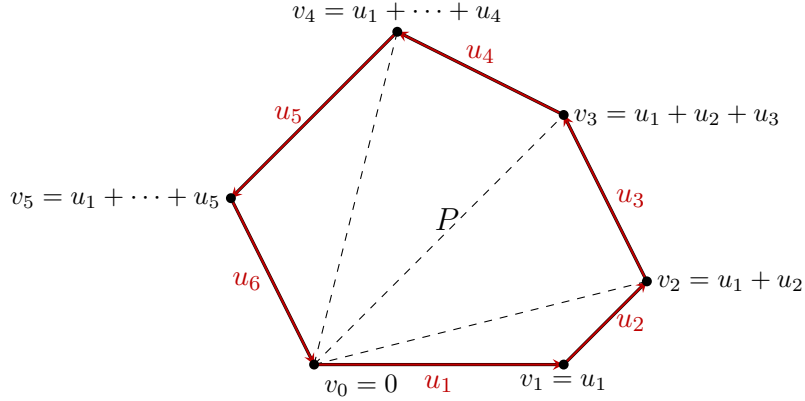
	\begin{proof}
		Denote by $u_i \in \ZZ^2$ the vector of length $\vol_1(P_{a_i})$ orthogonal to $a_i$, pointing counter-clockwise along the boundary of $P$, $i=1, \dots, 2N$. In view of $\sum_{i=1}^{2N}\delta_i a_i = 0$, we may assume without loss of generality that $P$ has one vertex at the origin and that the vectors $a_1, \dots, a_{2N}$ are sorted counter-clockwise with $a_1$ and $a_{2N}$ being the normals to the facets neighboring the origin. Then a triangulation of $P$ is given by
		\begin{align*}
			&\mathrm{conv}\{0, u_1, u_1+u_2\},\\
			&\mathrm{conv}\{0, u_1+u_2, u_1+u_2+u_3\}, \dots, \\
			&\mathrm{conv}\{0, \sum_{i=1}^{2N-1}u_i, \sum_{i=1}^{2N}u_i\},
		\end{align*}
		see \autoref{fig:volR2byFacets}, and the volume of $P$ can be calculated by
		\begin{align*}
			\vol(P) &= \sum_{j=2}^{2N} \vol\left(\mathrm{conv}\left\{0, \sum_{i=1}^{j-1}u_i, \sum_{i=1}^{j}u_i\right\}\right)\\
			&=\frac{1}{2}\sum_{j=2}^{2N} \det\left(\sum_{i=1}^{j-1}u_i, \sum_{i=1}^{j}u_i\right)\\
			&=\frac{1}{2}\sum_{j=2}^{2N}\sum_{i=1}^{j-1} \det\left(u_i, u_{j}\right) = \frac{1}{2}\sum_{1 \leq i < j \leq 2N} \det\left(u_i, u_{j}\right).
		\end{align*}
		Next, note that $u_i$ is obtained by rotating $a_i$ by $\frac{\pi}{2}$ counter-clockwise and rescaling such that
		\begin{align*}
			|u_i| = \vol_1(P_{a_i}) = |a_i| \delta_i.
		\end{align*}
		Consequently,
		\begin{align*}
			\det\left(u_i, u_{j}\right) = \delta_i\delta_j\det(a_i, a_j),
		\end{align*}
		concluding the proof.
	\end{proof}

	\begin{proof}[Proof of \autoref{mcor:tpR2}]
		The necessity of the conditions is clear from \eqref{eq:extr_degrees}, \autoref{mthm:descrTdm1} and \autoref{lem:volR2ByFacets}.
		
		For the sufficiency, note that the jumps in the 1-cells, $\delta_i$, are not changed by subtracting a constant function or $\eins_{\ZZ^2}$. By \autoref{mthm:descrT1inR2}, we therefore obtain a lattice polygon $P \in \polys(\ZZ^2)$ such that
		\begin{align*}
			\tcount_{P,1} = f - \eins_{\ZZ^2} - \frac{1}{2}\sum_{1\leq i < j < 2N} \delta_i\delta_j \det(a_i,a_j).
		\end{align*}
		The function $\tcount_P$ has the same jumps as $f$, and by \autoref{lem:volR2ByFacets}, the same integral. Hence, $f=\tcount_P$ holds as desired.
	\end{proof}
	
	We conclude the section with a curious observation that shows an additional symmetry property of the translation count in dimension $2$. It is also implied by the more general Corollary 5.4 in \cite{HigashitaniEtAl}.
	\begin{lemma}
		Let $P \in \polys(\ZZ^2)$ and choose $v \in \RR^2$. If $v$ is contained in a $2$-dimensional cell $C$ of $\hyp_P^\circ$ neighboring the origin, then
		\begin{align*}
			\frac{\tcount_P(v) + \tcount_P(-v)}{2} = \vol_2(P).
		\end{align*}
	\end{lemma}
	\begin{proof}
		Since $C$ is neighboring the origin, we have $\tcount_P(\pm v) = \tcount_P(\pm \varepsilon v)$ for all $\varepsilon > 0$. Consequently,
		\begin{align*}
			\Delta_{\pm v} \tcount_P(0) = \lim_{\varepsilon\searrow 0} \tcount_P(0) - \tcount_P(\pm \varepsilon v) = \tcount_P(0) - \tcount_P(\pm v).
		\end{align*}
		Noting that $\tcount_P(-v) = \tcount_{-P}(v)$, \autoref{lemma:tangentcones} implies that
		\begin{align*}
			\tcount_P(0) - \tcount_P(\pm v) = |\bd_{\pm v} P \cap \ZZ^2|,
		\end{align*}
		where we denote by $\bd_u P$ the positive boundary of $P$ with respect to a vector $u \in \RR^2 \setminus \{0\}$, that is, the set of all $x \in \bd P$ with outer normal $a \in (\RR^2)^\star$ such that $\pair{a}{u} > 0$. By our choice of $v$ (that is, since $v$ is not parallel to a facet of $P$), $\bd P = \bd_{v} P \cup \bd_{-v} P$ and the intersection consists of exactly two points. Therefore,
		\begin{align*}
			\frac{\tcount_P(v) + \tcount_P(-v)}{2} &= \tcount_P(0) - \frac{\Delta_v \tcount_P(0) + \Delta_{-v}\tcount_P(0)}{2}\\
			&=|P \cap \ZZ^2| - \tfrac{1}{2}\left(|\bd P \cap \ZZ^2| + 2\right) = \vol_2(P),
		\end{align*}
		where the last equality follows from Pick's theorem (see, e.g., \cite{Beck2015}*{Thm.~2.8}).
	\end{proof}

	\subsection*{Acknowledgments}
	The authors would like to thank the participants of the polytope workshop at SageDays Vienna 2026, Nicolas Hain, Matthias M\"uller, Martin Rubey and Valentina Taylor Cerra, for helping us to produce good \texttt{sagemath} code to support computational experiments. 
We thank Christian Stump for providing us access to frontier LLMs through his platform \href{https://math.sciencebench.ai/}{\texttt{sciencebench.ai}}.
	\subsection*{AI-disclosure statement}
	For the proofs in Section \ref{sec:kubert_lang}, we profited from ``conversations'' with an LLM. We obtained somewhat convoluted responses to our prompts, that we could not fully parse, but that were very helpful to inspire a rigorous proof by the authors.
	While every line in this paper is handwritten, we used LLMs for generating the illustrations.

\end{document}